\documentclass[reqno,a4paper,11pt]{amsart}

\usepackage{amsmath,amssymb,amsthm,mathrsfs,comment,enumerate,color}
\usepackage{bbm}

\numberwithin{equation}{section}

\newcommand{\1}{\mathbbm {1}}
\newcommand{\id}{\operatorname{Id}}
\newcommand{\dom}{\operatorname{dom}}

\newcommand{\ud}{\mathrm{d}}

\newcommand{\A}{\mathcal{A}}
\newcommand{\D}{\mathcal{D}}

\renewcommand{\mathbb}{\mathbbm}   
\newcommand{\N}{\mathbb{N}}
\newcommand{\R}{\mathbb{R}}

\newcommand{\E}{\mathbb{E}}

\newtheorem{defin}{Definition}[section]
\newtheorem{prop}[defin]{Proposition}
\newtheorem{theo}[defin]{Theorem}
\newtheorem{coroll}[defin]{Corollary}
\newtheorem{lemma}[defin]{Lemma}
\theoremstyle{definition}

\newtheorem{example}[defin]{Example}

\renewcommand{\le}{\leqslant}
\renewcommand{\ge}{\geqslant}
\renewcommand{\phi}{\varphi}

\newcommand{\scapro}[2]{\langle #1,#2\rangle}       

\renewcommand{\L}{{\mathcal L}}
\newcommand{\M}{\widehat{\mathcal M}}
\newcommand{\norm}[1]{\left\lVert #1 \right\rVert}   
\newcommand{\abs}[1]{\left\lvert #1 \right\rvert}    
\DeclareMathOperator{\cs}{\mathfrak Z}

\title[Cylindrical fractional Brownian motion]{Cylindrical fractional Brownian motion\\ in Banach spaces}

\author[E. Issoglio]{Elena Issoglio}
\address[E. Issoglio and M. Riedle]{Department of Mathematics\\
King's College\\
London WC2R 2LS\\
United Kingdom }
\email[E. Issoglio]{elena.issoglio@kcl.ac.uk}
\author[M. Riedle]{Markus Riedle}
\email[M. Riedle]{markus.riedle@kcl.ac.uk}
\thanks{The second named author acknowledges the EPSRC grant EP/I036990/1}



\begin{document}

\begin{abstract}
In this article we introduce  cylindrical fractional Brownian motions in Banach spaces and develop the related stochastic integration theory. Here a cylindrical fractional Brownian motion is understood in the classical framework of cylindrical random variables and cylindrical measures. The developed stochastic integral for deterministic operator valued integrands is based on a series representation of the cylindrical fractional Brownian motion, which is analogous to the Karhunen-Lo{\`e}ve expansion for genuine stochastic processes. In the last part we apply our results to study  the abstract stochastic Cauchy problem in a Banach space driven by cylindrical fractional Brownian motion.
\end{abstract}

\maketitle

\section{Introduction}

In the past decades, a wide variety of infinite dimensional stochastic equations have been studied, due to their broad range of applications in physics, biology, neuroscience and in numerous other areas.
A comprehensive study of stochastic evolution equations in Hilbert spaces driven by cylindrical Wiener processes, based on a semigroup approach, can be found in the monograph of Da Prato and Zabczyk \cite{da-prato_zabczyk92}. Various extensions and modifications have been studied, such as different types of noises as well as generalisations to Banach spaces. For the latter see for example Brze\'{z}niak \cite{brzezniak97} and van Neerven et al.\ \cite{van-neerven_et.al.07, JanLutz2005}.

Fractional Brownian motion (fBm) has become very popular in recent years as
driving noise in stochastic equations, in particular as an alternative to the classical Wiener noise. This is mainly due to  properties of fBms, such as long-term dependence, which leads to a \emph{memory} effect, and \emph{self-similarity}, features which show great potential for applications, for example in hydrology, telecommunication traffic, queueing theory and mathematical finance.
Since  fBms are not  semi-martingales, It\^o-type calculus cannot be applied. Several different stochastic integrals with respect to real valued fBm have been introduced in the literature, e.g. Wiener integrals for deterministic integrands, Skorohod integrals using Malliavin calculus techniques, pathwise integrals using generalised Stieltjes integrals or integrals based on rough path theory. For more details see e.g. \cite{biagini_et.al.08, mishura08, nualart06fBm} and references therein.


The purpose of this paper is to begin a systematic study of cylindrical fractional Brownian motion in Banach spaces and, starting from this, to build up a related stochastic calculus in Banach spaces with respect to cylindrical fBm.  Our approach is based on cylindrical measures and cylindrical random variables which enables us to develop a theory that does not require a  Hilbert space structure of the underlying  space because the cylindrical fBm is defined through  finite dimensional projections. 
We can characterise the cylindrical fBm by a series representation, which can be considered as the analogue of the Karhunen-Lo{\`e}ve expansion in the classical situation of genuine stochastic processes. This representation is exploited to define the stochastic integral of deterministic, operator valued integrands with respect to a cylindrical fBm. The stochastic integral is defined as a stochastic version of a \emph{Pettis} integral, as accomplished in van Neerven and Weis \cite{JanLutz2005} for Wiener processes and in  Riedle and van Gaans \cite{riedle_vanGaans09} for L\'evy processes. As  the integrand is deterministic, the integral process is Gaussian  and therefore it is characterised by its covariance operator.

We apply our theory to a class of parabolic stochastic equations in Banach spaces of the form
\[
\ud Y(t)= AY(t)\,\ud t +C\,\ud B(t),
\]
where  $B$ is a cylindrical fBm in a separable Banach space $U$, $A$ is a generator  of a strongly continuous semigroup in a separable Banach space $V$ and $C$ is a linear and continuous operator from $U$ to $V$.  We give necessary and sufficient conditions for the existence and uniqueness of a weak solution,  which is a genuine stochastic process in the Banach space $U$. For comparison, we apply our methods to an  example often considered in the literature and typically formulated in a Hilbert space setting.

The systematic approach adopted in this paper goes back to Kallianpur and Xiong \cite{KallianpurXiong} and to Metivier and Pellaumail  \cite{metivier_pellaumail80}, who treated the cylindrical Wiener case and the cylindrical martingale case, respectively. In this paper we consider an extension beyond the martingale case, since fractional Brownian motion is not a semi-martingale.
Our methodology, based on cylindrical measures and cylindrical random variables, has the advantage that it is \emph{intrinsic} in the sense that it does not require the construction of a larger space in which the cylindrical noise exists as a genuine stochastic process. Due to the connection between cylindrical measures and the theory of geometry of Banach spaces, our methodology  relates the study of fBm and stochastic differential equations driven by fBm to other areas of mathematics, such as operator theory, functional analysis and harmonic analysis, therefore providing a wider range of tools and techniques.

Our long-term aim is to study general stochastic equations in Banach spaces driven by cylindrical fBms, which involves stochastic integration for random integrands. We are inspired by the paper of van Neerven et al.\ \cite{van-neerven_et.al.07} in which they deal with the Wiener case. Here, the approach is based on a two-sided decoupling inequality which enables the authors to define the stochastic integral for random integrands by means of
the integral for deterministic integrands. The latter is introduced in  van Neerven and Weis  \cite{JanLutz2005}, and we hope that our present work will play an analogous role for equations driven by fractional Brownian motions.

Only a few works deal with fBm in Banach spaces and related stochastic integration theory. Brze\'{z}niak et al.\ \cite{brzezniak_et.al.12} consider abstract Cauchy problems in Banach spaces driven by cylindrical Liouville fBm. It is shown that for $H<1/2$ the theory for Liouville fBm is equivalent to the one for fBm, while for $H>1/2$ the space of integrable functions is slightly different. In our paper we extend their results related to the Cauchy problem, as we consider mild and weak  solutions and we  obtain necessary and sufficient conditions for the existence of a solution. In contrast to \cite{brzezniak_et.al.12}, we do not assume any further regularity for the diffusion operator $C$, and therefore we keep the irregular character of the cylindrical noise in the space where the equation is considered. Note however, that in some special cases the authors in \cite{brzezniak_et.al.12} get around this restriction by means of  interpolation techniques.
Furthermore, our approach enables us to guarantee the existence of a solution for $H>1/2$ without any further constraints, whereas the results in this case in  \cite{brzezniak_et.al.12} are restricted either
 to analytic semigroups or to Banach spaces of type larger than 1.
Another approach for the study of an evolution equation driven by a fractional Brownian motion is considered by Balan in \cite{balan11}. The author considers a stochastic heat equation with infinite dimensional fractional noise by using Malliavin calculus, but her approach is strictly limited to a Sobolev-space context.

In the special case of Hilbert spaces, quite some literature on stochastic evolution equations  with fBm noise can be found -- see amongst others Grecksch and Anh \cite{grecksch_anh99},  Duncan and coauthors in a series of papers \ \cite{ duncan_et.al.06, duncan_et.al.05, duncan_et.al.09, duncan_et.al.02}, Tindel et al.\ \cite{tindel_et.al.03}, Maslowski and Nualart \cite{maslowski_nualart03},  Gubinelli et al.\ \cite{gubinelli_et.al.06}.
When restricted to the Hilbert space case, our approach is to some extent similar to the one in \cite{duncan_et.al.06,  duncan_et.al.09, duncan_et.al.02}.
But our method has the advantage of providing not only sufficient, but also necessary conditions for the existence of a solution, which turns out to be both mild and weak.
Tindel et al.\ in \cite{tindel_et.al.03}, who also provide necessary and sufficient conditions for the existence of
a solution, derive their results under a spectral gap assumption on the semigroup, which is assumed to be self-adjoint. Our approach enables us to avoid such kind of restrictive assumptions.

The paper is structured as follows. Section \ref{sc: preliminaries} contains a brief overview of cylindrical measures and cylindrical processes in Banach spaces, with emphasis on cylindrical Gaussian processes. In Section \ref{sc: wiener integrals for hilbert space valued fcs} we recall the construction of the Wiener integral for real valued fBm
and its relation to fractional integral and derivative operators. In Section \ref{sc: Cylindrical fractional Brownian motion} we introduce  cylindrical fBms in  separable Banach spaces and provide a characterisation in terms of a series representation. We illustrate our
notion of fBm by several examples, such as \emph{anisotropic} fBm, i.e.\ spatially non-symmetric noise, and we give conditions under which such cylindrical noises are genuine fBms in the underlying space. Section \ref{sc: integration} is dedicated to the construction and the study of the  stochastic integral in a Banach space. In Section \ref{sc: cauchy
problem} we use this integral  to construct the fractional Ornstein-Uhlenbeck process as the mild and weak solution of a abstract stochastic Cauchy problem in a Banach space. Finally, in Section \ref{sc: application} we consider the special case of the stochastic heat equation with fractional noise in a Hilbert space and compare our results with the existing literature.


\section{Preliminaries}\label{sc: preliminaries}

Throughout this paper, $U$ indicates a separable Banach space over $\R$ with norm $\|\cdot\|_U$. The topological dual   of $U$ is denoted by $U^*$ and the algebraic one by $U'$. For $u^*\in U^*$ we indicate the dual pairing  by $\langle  u, u^* \rangle $. If $U$ is a Hilbert space we identify the dual space $U^\ast$ with $U$.  The Borel $\sigma$-algebra on a Banach space $ U$ is denoted by $\mathcal B(U)$. If $V$ is another Banach space then $\L(U,V)$ denotes the space of bounded, linear operators from $U$ to $V$ equipped with the operator norm topology.

For a measure space $(S,{\mathcal S},\mu)$ we denote by $L^p_\mu(S;U)$, $p\ge 0$, the space of equivalence classes of measurable functions $f:S\to U$ with $\int \norm{f(s)}_U^p\,\mu(\ud s)<\infty$. If $S\in {\mathcal B}(\R)$ and $\mu$ is the Lebesgue measure we use the notation $L^p(S;U)$.

Next we recall some notions about cylindrical measures and cylindrical random variables as it can be found in Badrikian \cite{Badrikian} or Schwartz \cite{Schwartz}.
 Let $\Gamma$ be a subset of $U^*$, $n\in \N $, $u_1^* ,\ldots , u_n^* \in \Gamma $ and $B \in \mathcal B( \R^n )$.
A set of the form
\[
\cs(u^*_1 ,\ldots , u^*_n ; B) := \{u \in U : ( \langle u, u^*_1 \rangle,\ldots, \langle u, u^*_n\rangle ) \in B\},
\]
is called a {\em cylindrical set}. We denote by $\mathcal Z(U,\Gamma)$ the set of all cylindrical sets in $U$ for a given $\Gamma$.
 It turns out this is an
 \emph{algebra}. Let  $\mathcal C(U,\Gamma)$ be the generated \emph{$\sigma$-algebra}. When $\Gamma= U^*$ the notation is
$\mathcal Z(U)$ and  $\mathcal C(U)$, respectively. If $U$ is separable
then both the Borel  $\sigma$-algebra $\mathcal B(U )$ and the cylindrical $\sigma$-algebra $\mathcal C(U )$ coincide. \\
 A  function $\mu:\mathcal Z(U) \to [0,\infty]$ is called  a
\emph{cylindrical measure} on $\mathcal Z(U)$ if for each finite subset $\Gamma\subseteq U^*$ the restriction of $\mu$ to
the $\sigma$-algebra $\mathcal C(U,\Gamma)$ is a measure. It is called \emph{finite} if $\mu(U)$ is finite and \emph{cylindrical probability measure} if  $\mu(U) =1$.

For every function $f:U\to{\mathbb C}$ which is measurable with respect to
${\mathcal Z}(U,\Gamma)$ for a finite subset $\Gamma\subseteq U^\ast$ the integral $\int
f(u)\,\mu(\ud u)$ is well defined as a complex valued Lebesgue integral if it
exists. In particular, the characteristic function $\phi_\mu:U^\ast\to{\mathbb C}$ of a
finite cylindrical measure $\mu$ is defined by
\begin{align*}
 \phi_{\mu}(u^\ast):=\int_U e^{\imath\, \scapro{u}{u^\ast}}\,\mu(\ud u)\qquad\text{for all }u^\ast\in
 U^\ast.
\end{align*}

Let $(\Omega, \A,P)$ be a probability space. A {\em cylindrical random variable $ Z$ in $U$} is a linear and continuous map
\[
   Z:U^*\to L_P^0(\Omega;\R),
\]
where $L_P^0(\Omega;\R)$ is equipped with the topology of convergence in probability. The characteristic function of a cylindrical random  variable $Z$ is defined by
\begin{align*}
 \phi_Z:U^\ast \to {\mathbb C}, \qquad \phi_Z(u^\ast)=E[\exp(\imath\, Zu^\ast)].
\end{align*}
A cylindrical process in $U$ is a family $( Z(t):\, t\ge 0)$ of cylindrical random variables in $U$.

Let $Z:U^*\to L_P^0(\Omega;\R)$ be a cylindrical random variable in $U$. If $\cs=\cs(u_1^\ast,\dots, u_n^\ast;B)$ is a cylindrical set for
$u^\ast_1,\dots, u^\ast_n\in U^\ast$ and $B\in {\mathcal B}(\R^n)$, we obtain a cylindrical probability measure $\mu$ by the prescription
\begin{align*}
  \mu(\cs):=P((Zu^\ast_1,\dots, Zu^\ast_n)\in B).
\end{align*}
We call $\mu$ the {\em cylindrical distribution of $Z$} and the
characteristic functions $\phi_\mu$ and $\phi_Z$ of $\mu$ and $Z$
coincide. Conversely, for every cylindrical measure $\mu$ on
${\mathcal Z}(U)$ there exist a probability space $(\Omega,\A,P)$ and a cylindrical random variable $Z:U^\ast\to L^0_P(\Omega;\R)$ such
that  $\mu$ is the cylindrical distribution of $Z$.

A cylindrical probability measure $\mu$ on $\mathcal Z(U)$ is called
 \emph{Gaussian} if the image measure $\mu \circ (u^*)^{-1}$ is a Gaussian measure on $\mathcal B(\R)$ for all $u^*\in U^*$.
The characteristic function $\varphi_\mu:U^*\to \mathbb C$ of a
  Gaussian cylindrical measure $\mu$ is of the form
\begin{equation}\label{equation charact fc cylin meas}
\varphi_\mu(u^*)= \exp\left( \imath\, m(u^*) - \tfrac{1}{2} s(u^*)\right)
\qquad\text{for all }u^\ast\in U^\ast,
\end{equation}
where the mappings $m:U^*\to \R$ and $s:U^*\to \R_+$ are given by
\[
m(u^*)=\int_U \langle u, u^*\rangle \, \mu(\ud u), \quad s(u^*)=\int_U \langle u, u^*\rangle^2 \,\mu(\ud u) - m(u^*)^2.
\]
Conversely, if  $\mu$  is a cylindrical measure with characteristic function of the form \eqref{equation charact fc cylin meas}
for a linear functional  $m:U^*\to \R$ and a quadratic form $s:U^*\to \R_+$, then $\mu$ is a Gaussian cylindrical measure.

For a  Gaussian cylindrical measure $\mu$
with characteristic function of the form \eqref{equation charact fc cylin meas} one defines the covariance operator $Q: U^*\to (U^*)'$   by
\[
\left(Q u^*\right)v^*= \int_U\langle u, u^*\rangle \langle u,v^*\rangle \,\mu(\ud u)-  m(u^*)m(v^*)\qquad\text{for all }u^\ast, v^\ast\in U^\ast.
\]
On the contrary to Gaussian Radon measures, the covariance operator might take values only in the algebraic dual of $U^*$, that is the linear map $Qu^\ast:U^\ast\to \R$ might be not continuous for some $u^\ast\in U^\ast$. However often we exclude this rather general situation by requiring at least that $Qu^\ast$ is norm continuous, that is $Q:U^\ast\to U^{\ast\ast}$. Note that in this situation the characteristic function $\phi_\mu$ of $\mu$ in \eqref{equation charact fc cylin meas} can be written as
\begin{align*}
\phi_\mu(u^*)= \exp\left( \imath\, m(u^*) - \tfrac{1}{2} \scapro{u^\ast}{Qu^\ast}\right)
\qquad\text{for all }u^\ast\in U^\ast.
\end{align*}
A cylindrical random variable $Z:U^\ast\to L_P^0(\Omega;\R)$ is called {\em Gaussian} if its cylindrical distribution is Gaussian.
 Since we require from the cylindrical random variable $Z$ to be  continuous it follows that its characteristic function $\phi_Z:U^\ast\to {\mathbb C}$ is continuous. The latter occurs if and only if the covariance operator $Q$ maps to $U^{\ast\ast}$.

\section{Wiener integrals for Hilbert space valued integrands}
\label{sc: wiener integrals for hilbert space valued fcs}
In the following we recall the construction of the Wiener integral with respect to a real valued fractional Brownian motion for integrands which are  Hilbert space valued deterministic functions. For real valued integrands the construction is accomplished for example in
\cite{biagini_et.al.08} and for Hilbert space valued integrands in
\cite{duncan_et.al.06,duncan_et.al.09,pasik-duncan_et.al.06}.

We begin with recalling the definition of a fractional Brownian motion (fBm)  and for later purpose, we introduce it in $\R^n$. A Gaussian process $( b(t):\, t\ge 0)$ in $\R^n$ is
a {\em fractional Brownian motion with Hurst parameter $H\in(0,1)$} if there exists a matrix $M\in\R^{n\times n}$ such that
\begin{align*}
   E\big[\scapro{\alpha}{b(s)}\big]=0,\qquad\quad
   E\big[\scapro{\alpha}{b(s)}\scapro{\beta}{b(t)}\big]= \scapro{M\alpha}{\beta} R(s,t)
\end{align*}
for all $s,t\ge 0$ and $\alpha,\beta \in\R^n$, where
\[
R(s,t):= \tfrac{1}{2}\left(s^{2H}+t^{2H}-|s-t|^{2H}\right)
\qquad\text{for all }s,t\ge 0.
\]
The matrix $M=(m_{i,j})_{i,j=1}^n$ is called the {\em covariance matrix} of
the fBm $\big((b_1(t),\dots, b_n(t)):\,t\ge 0)$  in $\R^n$ since it follows
\begin{align*}
  m_{i,j}=E\big[b_i(1)b_j(1)\big] \qquad\text{ for all $i,j=1,\dots, n$.}
\end{align*}
Thus, $M$ is a positive and symmetric matrix.
If $M=\id$ then $b$ is called {\em standard fractional Brownian motion}.
It follows from Kolmogorov's continuity theorem by the Garsia-Rodemich-Rumsey
inequality, that there exists a version of a fBm with H{\"o}lder continuous paths of
any order smaller than $H$.

We fix for the complete work the Hurst parameter and assume $H\in (0,1)\setminus\{\tfrac{1}{2}\}$.
The covariance function has an integral representation given by
\begin{equation}\label{equation integral representation of R(t,s)}
R(s,t)=\int_0^{s\wedge t}\kappa(s,u)\kappa(t,u) \, \ud u
\qquad\text{for all }s,t\ge 0,
\end{equation}
where the kernel $\kappa$ has different expressions depending on the Hurst parameter. If $H>\tfrac{1}{2}$ then
\[
\kappa(t,u)=b_H u^{1/2-H} \int_u^t (r-u)^{H-3/2} r^{H-1/2}\, \ud r
\qquad\text{for all }0\le u< t,
\]
where $b_H=(H(2H-1))^{1/2}(\beta(2-2H, H-1/2))^{-1/2}$ and $\beta$ denotes the Beta function.
If $H<\frac{1}{2}$, we have
\begin{align*}
\kappa(t,u)=& \,b_H \Big( \left(\tfrac{t}{u}\right)^{H-1/2}( {t}-{u})^{H-1/2}\\
 &- \left(H-\tfrac{1}{2}\right) u^{1/2-H} \int_u^t (r-u)^{H-1/2} r^{H-3/2}\, \ud r\Big)\qquad\text{for all }0\le u< t,
\end{align*}
where $b_H=[2H/((1-2H)\beta(1-2H, H+1/2))]^{1/2}$.

Let $X$ be a separable Hilbert space with scalar product $[\cdot, \cdot]$. A {\em simple} $X$-valued function $f:[0,T]\to X$ is of the form
\begin{align}\label{eq.simplevarphi}
f(t)=\sum_{i=0}^{n-1} x_i \1_{[t_i,t_{i+1})}(t)
\qquad \text{for all }t\in [0,T],
\end{align}
where $x_i\in X$, $0=t_0<t_1<\cdots< t_{n}=T$ and $n\in\N$. The space of all simple, $X$-valued functions is denoted by $\mathcal E$ and it is  equipped with an inner product defined by
\begin{equation}\label{equation def scalar product on E H calligraphic}
\left\langle \sum_{i=0}^{m-1} x_i \1_{[0,s_{i})} \, , \sum_{j=0}^{n-1} y_j \1_{[0,t_j)} \right \rangle_{\mathcal M}:= \sum_{i=0}^{m-1} \sum_{j=0}^{n-1}  [x_i ,y_j] R(s_i,t_j).
\end{equation}
Thus, $\mathcal E$ is a pre-Hilbert space. We denote
 the closure of $\mathcal E$ with respect to $\langle  \cdot, \cdot \rangle_{\mathcal M}$ by $\mathcal M$.

Let $(b(t):\,t\ge 0)$  be a real valued fractional Brownian motion with
Hurst parameter $H$. For a simple, $X$-valued function $f:[0,T]\to X$ of the form \eqref{eq.simplevarphi}  we define the {\em Wiener integral} by
\begin{equation*}
\int_0^T f\, \ud b  := \sum_{i=0}^{n-1} x_i \big( b (t_{i+1})-b (t_i)\big).
\end{equation*}
 The integral $\int f\, \ud b $  is a random variable in $L^2_P(\Omega;X)$ and the map $f\mapsto \int f\, \ud  b$ defines an isometry between $\mathcal E$  and $L^2_P(\Omega;X)$, since
\begin{equation}\label{equation isometry step function in Hilbert space}
\left \|\int_0^T f \,\ud  b  \right \|^2_{L^2_P}= \|f \|_{\mathcal M}^2.
\end{equation}
Consequently, we can extend the mapping $f\mapsto \int f\, \ud  b$ to the space ${\mathcal M}$ and
 the extension still satisfies the isometry \eqref{equation isometry step function in Hilbert space}.

There is an alternative prescription of the space ${\mathcal M}$ of possible integrands. For that purpose, we introduce the linear operator
$K^*: \mathcal E\to L^2([0,T];X)$,
which is defined for all $t\in [0,T]$ in case $H<\tfrac{1}{2}$ by
\begin{align*}
  (K^* f)(t)&:=f(t)\kappa(T,t)+\int_t^T (f(s)-f(t)) \frac{\partial\kappa}{\partial s}(s,t)\, \ud s,
\intertext{
and in case $H>\tfrac{1}{2}$  by}
  (K^* f)(t)&:=\int_t^T f(s) \frac{\partial\kappa}{\partial s}(s,t)\, \ud s .
\end{align*}
The integrals appearing on the right-hand side are both Bochner integrals. Since the operator $K^\ast$ satisfies
\begin{align}\label{eq.KastScaPro}
  \scapro{K^\ast f}{K^\ast g}_{L^2}=\scapro{f}{g}_{\mathcal M}
  \qquad\text{for all }f,\,g\in {\mathcal E},
\end{align}
it can be extended to an isometry $K^\ast$ between $\mathcal M$ and
 $L^2([0,T];X)$. Together with (\ref{equation isometry step function in Hilbert space}) we obtain
\begin{equation}\label{equation isometry wiener integral Hilbert}
 \left\|\int_0^T f\, \ud   b  \right \|_{L^2_P}^2=
 \|K^*f\|^2_{L^2} =\|f\|_{\mathcal  M }^2
 \qquad\text{for all }f\in\mathcal M.
 \end{equation}
The operator $K^\ast$ can be rewritten  using the notion of fractional integrals and derivatives. For this purpose, define for $\alpha>0$ the {\em fractional integral operator} $I^\alpha_{T-}:
L^2([0,T];X)\to L^2([0,T];X)$  by
\begin{align*}
\left(I^\alpha_{T-} f\right)(t):= \frac{1}{\Gamma(\alpha)} \int_t^T (s-t)^{\alpha-1} f (s) \, \ud s\qquad\text{for all } t\in [0,T].
\end{align*}
Young's inequality
guarantees that $I^\alpha_{T-} f \in L^2([0,T];X)$ and that the operator
$I^\alpha_{T-}$ is bounded on $L^2([0,T];X)$. We define the space
\begin{align*}
  H_{T-}^\alpha([0,T];X):=I^\alpha_{T-}(L^2([0,T];X))
  \end{align*}
and equip it with the norm
\begin{align*}
\norm{I^\alpha_{T-} f}_{H_{T-}^\alpha}:=\norm{f}_{L^2}
\qquad\text{for all } f\in L^2([0,T];X).
\end{align*}
It follows that the space
$H_{T-}^\alpha([0,T];X)$ is a Hilbert space and it is continuously embedded in $L^2([0,T];X)$.

For $\alpha\in(0,1)$ the {\em fractional differential operator}
$D^\alpha_{T-}: H_{T-}^\alpha([0,T];X)\to L^2([0,T];X)$ is defined
by
\begin{align*}
  (D^\alpha_{T-}  f)(t):=\frac{1}{\Gamma(1-\alpha)}
  \left(\frac{f(t)}{(T-t)^\alpha}+ \alpha\int_t^T
   \frac{f(t)-f(s)}{(s-t)^{\alpha+1}}\,\ud s\right)
\end{align*}
for all $t\in [0,T]$.
The fractional integral and differential operators obey the  inversion formulas
\begin{align*}
  I^\alpha_{T-}(D^\alpha_{T-} f)=f\qquad\text{for all }f\in H^\alpha_{T-}([0,T];X),
\end{align*}
and
\begin{align*}
  D^\alpha_{T-}(I^\alpha_{T-} f)=f\qquad\text{for all }f\in L^1([0,T];X).
\end{align*}


Let $p^{H-1/2}$ denote the function $p^{H-1/2}(t)=t^{H-1/2}$ for all $t\in [0,T]$. The operator $K^\ast$ can be rewritten
in the case $H>1/2$ as
\begin{align}
  (K^\ast f)(t)&=b_H \Gamma\left(H-\tfrac{1}{2}\right)t^{1/2-H}
  I_{T-}^{H-1/2}\left( p^{H-1/2} f \right)(t)\label{eq.KastH>0.5}
\intertext{for all $t\in[0,T]$ and in the case of $H<1/2$ in the form}
  (K^\ast f)(t)
  &=b_H\Gamma\left(H+\tfrac{1}{2}\right) t^{1/2-H} D_{T-}^{1/2-H}\left(p^{H-1/2} f \right)(t).
  \label{eq.KastH<0.5}
\end{align}
It can be seen from \eqref{eq.KastH>0.5} that $\mathcal M$ contains distribution for $H>\tfrac{1}{2}$. Thus it became standard to restrict the space ${\mathcal M}$ in this case, see for example \cite{biagini_et.al.08, duncan_et.al.09,pipiras_taqqu00}. It turns out that an appropriate choice is the function space
\[
\vert\mathcal M \vert:=\left\{f:[0,T]\to X:\,
\int_0^T\int_0^T \|f(s)\| \|f(t)\| \vert s-t\vert^{2H-2} \, \ud s \, \ud t < \infty\right\},
\]
equipped with the norm
\begin{align*}
\norm{f}_{\abs{\mathcal M}}^2:=
  H(2H-1) \int_0^T\int_0^T \|f(s)\| \|f(t)\| \vert s-t\vert^{2H-2} \, \ud s \, \ud t.
\end{align*}
The space $\abs{\mathcal M}$ is complete and it is  continuously embedded in $\mathcal M$. The proof of this fact is analogous to the real valued case, see e.~g.~\cite[Pro.2.1.13]{biagini_et.al.08}.
If  $H>\frac{1}{2}$ then the  covariance function $R$ is differentiable with
\begin{align*}
  \frac{\partial^2 R}{\partial s \partial t}(s,t)= H(2H-1) \abs{s-t}^{2H-2}
  \qquad\text{for all }s,t\ge 0,
\end{align*}
and we can rewrite \eqref{equation def scalar product on E H calligraphic} as
\begin{align}\label{eq.scalarMH>0.5}
  \scapro{f}{g}_{\mathcal M}=H(2H-1)\int_0^T\int_0^T [f(s),g(t)]\abs{s-t}^{2H-2}\,\ud s\,\ud t
\end{align}
for all simple functions $f,g \in \mathcal E$. Since $\mathcal E$ is dense in  $\abs{\mathcal M}$, equation \eqref{eq.scalarMH>0.5} is true for
all $f,g \in \vert\mathcal M \vert$, see \cite[Eq.(2.14)]{duncan_et.al.09}.

%
We summarise the two cases by defining
\begin{equation}
 \widehat{ \mathcal M }:=
\begin{cases}
 \mathcal M & \text{if } H\in (0,1/2),\\
 \vert\mathcal M \vert & \text{if } H\in (1/2,1).
\end{cases}
\end{equation}
Recall that $\M$ is  a Banach space and the operator $K^\ast$ satisfies
\begin{align}\label{eq.KisoMmod}
  \norm{K^\ast f}_{L^2}\le c \norm{f}_{\M}
  \qquad\text{for all }f\in \M
\end{align}
for a constant $c>0$. Inequality \eqref{eq.KisoMmod} follows
 from \eqref{equation isometry wiener integral Hilbert} and, if $H>\tfrac{1}{2}$, from
   the continuous embedding $\abs{\mathcal M} \hookrightarrow {\mathcal M}$. If $H<\tfrac{1}{2}$
   we can choose  $c=1$.

In the sequel,  we collect some properties of the spaces  $\mathcal M$ and $\abs{\mathcal M}$. Recall that the time interval $[0,T]$ is fixed. In our first
result the coincidence of the spaces are well known, whereas we are only aware that the
equivalence of the norms is stated in \cite{brzezniak_et.al.12} but without a proof.
\begin{prop}\label{prop.eqMandH}
  For $H<\tfrac{1}{2}$ the spaces $\mathcal M$ and $H_{T-}^{1/2-H}([0,T];X)$ coincide and
  the norms are equivalent.
\end{prop}
\begin{proof}
The fact that the spaces coincide is shown in  \cite[Pro.6]{AlosMazetNualart}.
The proof of the equivalence of the norms is based on the following relation,
which can be found in the proof of  \cite[Pro.6]{AlosMazetNualart}:
\begin{align}\label{eq.decompNualart}
  K^\ast f= a \left(D_{T-}^{1/2-H}f\right)+Rf  \qquad\text{for all }f\in {\mathcal M},
\end{align}
where $a:=b_H\Gamma(H+\tfrac{1}{2})$ and $R:L^2([0,T];X)\to L^2([0,T];X)$ is
a linear and continuous operator.
Since $H_{T-}^{1/2-H}([0,T];X)$ is continuously
embedded in $L^2([0,T];X)$ there exists a  constant $c>0$
such that for each $f\in \mathcal M$ we have
\begin{align*}
  \norm{f}_{\mathcal M}=\norm{K^\ast f}_{L^2}
  &\le a\norm{D_{T-}^{1/2-H}f}_{L^2}+ \norm{Rf}_{L^2}
  \le \left(a+ c\norm{R}\right)\norm{f}_{H_{T-}^{1/2-H}}.
\end{align*}
On the other hand,  the Hardy-Littlewood inequality in weighted spaces  guarantees that $\mathcal M$ is continuously embedded in $L^{1/H}([0,T];X)$. More specifically,
by choosing $p=2, \alpha=\tfrac{1}{2}-H, m=0, q=\tfrac{1}{H}, \mu=2\alpha, \nu=q\alpha$
in  \cite[Th.5.4]{samko_et.al.93}, 
we obtain for $f\in \mathcal M$
\begin{align}\label{eq.Hardy-Littlewood-weighted}
 \norm{f}_{L^q}
& =\left(\int_0^T \norm{f(t)}^q\,\ud t\right)^{1/q}\notag\\
&= \left(\int_0^T t^\nu\norm{t^{-\alpha} f(t)}^q\,\ud t\right)^{1/q}\notag\\
&= \left(\int_0^T t^\nu\norm{(I^\alpha_{T-} D^\alpha_{T-} p^{-\alpha} f)(t)}^q\,\ud t\right)^{1/q}\notag\\
&\le c\left(\int_0^T t^\mu\norm{(D^\alpha_{T-} p^{-\alpha} f)(t)}^p\,\ud t\right)^{1/p}\notag\\
&=c (b_H \Gamma(H+\tfrac{1}{2}))^{-1/p} \norm{K^\ast f}_{L^2}\notag\\
&=c (b_H \Gamma(H+\tfrac{1}{2}))^{-1/p}\norm{f}_{\mathcal M},
\end{align}
for a constant $c>0$.
Consequently, together with  the continuous
embedding of $L^{1/H}([0,T];X)$ in $L^2([0,T];X)$, it follows from
\eqref{eq.decompNualart} that each $f\in \mathcal M$ satisfies
\begin{align*}
 a \norm{f}_{H_{T-}^{1/2-H}}
  \le \norm{K^\ast f}_{L^2} + \norm{Rf}_{L^2}
  = \left( 1+ c (b_H \Gamma(H+\tfrac{1}{2}))^{-1/2} \norm{R}\right) \norm{f}_{\mathcal M},
\end{align*}
which completes the proof.
\end{proof}
\begin{prop}  \label{pro.extrefinM}
For every $t\in [0,T]$ there exists a constant $c_t>0$ such
that each $f\in\M$ obeys:
  \begin{itemize}
  \item[(a)] $\1_{[0,t]}f\in \M$ and
    $  \norm{\1_{[0,t]}f}_{\mathcal M}\le c_t \norm{f}_{\mathcal M}$.
  \item[(b)]  $\1_{[0,t]}f(t-\cdot)\in \M$ and $\norm{\1_{[0,t]}f(t-\cdot)}_{\mathcal M}=\norm{\1_{[0,t]}f}_{\mathcal M}$.
  \end{itemize}
\end{prop}
\begin{proof}
If $H>\tfrac{1}{2}$, both properties (a) and (b) follow from \eqref{eq.scalarMH>0.5}
with $c_t=1$ for all $t\in [0,T]$.
If $H<\tfrac{1}{2}$, note that it is known for $f\in H_{T-}^{1/2-H}([0,T];X)$ that
 $\1_{[0,t]}f$ and $\1_{[0,t]}f(t-\cdot)$ are in $H_{T-}^{1/2-H}([0,T];X)$, see \cite[Th.13.9, Th.13.10, Re.13.3]{samko_et.al.93} or \cite[Le.2.1, Le.2.2]{brzezniak_et.al.12}. Furthermore,  there  exists a constant $a_t>0$ such that
\begin{align*}
  \norm{\1_{[0,t]}f}_{H_{T-}^{1/2-H}}\le a_t\norm{f}_{H_{T-}^{1/2-H}}.
\end{align*}
Thus Proposition \ref{prop.eqMandH} implies  part (a) and $\1_{[0,t]}f(t-\cdot)\in \M$.
To show the norm equality in part (b), note the identity
 \begin{align}\label{eq.scalarMMarchaud}
   \scapro{g}{h}=e_H\left\langle {\mathbb D}_{-}^{1/2-H}g,\,{\mathbb D}_{+}^{1/2-H}h\right\rangle_{L^2}
   \qquad\text{for all }g,h\in {\mathcal M},
 \end{align}
where $e_H$ denotes a constant depending only on $H$, see  \cite[page 286]{Nualart2006}.
Here ${\mathbb D}_{\pm}^\alpha$ denote the right-sided/left-sided Weyl-Marchaud fractional derivatives  defined by
\begin{align*}
  {\mathbb D}_{\pm}^\alpha g(r):=\frac{\alpha}{\Gamma(1-\alpha)}\int_0^\infty
     \frac{g(r)-g(r\mp s)}{s^{1+\alpha}}\,\ud s
     \qquad\text{for all }r\in \R.
\end{align*}
It follows from \eqref{eq.scalarMMarchaud} that
\begin{align*}
  \norm{\1_{[0,t]} f(t-\cdot)}_{\mathcal M}^2
  &=e_H\big\langle\big({\mathbb D}_{-}^{1/2-H} \1_{[0,t]}f(t-\cdot)\big)(\cdot),\,\big({\mathbb D}_{+}^{1/2-H}\1_{[0,t]}f(t-\cdot)\big)(\cdot)\big\rangle\\
&=e_H\big\langle\big({\mathbb D}_{+}^{1/2-H}\1_{[0,t]}f\big)(t-\cdot),\,\big({\mathbb D}_{-}^{1/2-H}\1_{[0,t]}f\big)(t-\cdot)\big\rangle\\
&= \norm{\1_{[0,t]}f}_{\mathcal M}^2,
\end{align*}
which completes the proof.
\end{proof}

In the following we prove a technical result  that links the real case, that is $X=\R$ in the above, with the Hilbert case. For this reason we will stress the dependence on the underlying space by writing either $K^*_\R$ or $K^*_X$. Analogous notation will be adopted for the space $\M$.
\begin{prop}\label{prop link operator K and K}\hfill
\begin{itemize}
\item[(a)] Let $f$ be in $\M_\R$ and $x\in X$.
 Then
 \begin{align*}
   F:[0,T]\to X,\qquad F(t)=x\,f(t),
 \end{align*}
defines an element in $ \M_X$  satisfying $(K^*_X F)(\cdot) = x(K^*_\R f)(\cdot)$.
\item[(b)] Let $F$ be in $\M_X$ and $x\in X$. Then
\begin{align*}
   f :[0,T]\to\R\,\qquad f (t)=[F(t),x],
 \end{align*}
defines an element in $\M_\R$ satisfying $\scapro{K^*_X F(\cdot)}{x} = ( K^*_\R  f ) (\cdot)$.
\end{itemize}
\end{prop}

\begin{proof}
We prove only part (a) as part (b) can be done analogously.
  If $H<\tfrac{1}{2}$  then by Proposition \ref{prop.eqMandH} there exists $\phi_f\in L^2([0,T];\R)$ such that $f= I^{1/2-H}_{T-} \phi_f$. Since $x\phi_f\in L^2([0,T];X)$ and $F= x I^{1/2-H}_{T-} \phi_f=  I^{1/2-H}_{T-}  x\phi_f$, it follows that $F\in \M_X$. If $H>\frac{1}{2}$, the
  assumption $f\in |M|_{\R}$ implies $F\in |M|_X$. In both cases, the very definition of $K^*_X$ and $K^\ast_{\R}$ shows  $ K^*_X F= x K^*_\R f$.
\end{proof}

\section{Cylindrical fractional Brownian motion}\label{sc: Cylindrical fractional Brownian motion}

We define cylindrical fractional Brownian motions in a separable Banach space $U$ by following the classical approach of cylindrical processes. In the same way, one can introduce cylindrical Wiener processes, see for instance \cite{KallianpurXiong, metivier_pellaumail80, riedle11},  and recently, this approach has been accomplished in \cite{appelbaum_riedle10} to give the first systematic treatment of cylindrical L{\'e}vy processes.
 \begin{defin}\label{definition fBm cylindrical}
 A  cylindrical process $(B(t):\, t\ge 0)$ in $U$ is a {\em cylindrical fractional Brownian motion with Hurst parameter $H\in(0,1)$} if
  \begin{itemize}
  \item[(a)]  for any $u_1^*, \ldots,u_n^* \in U^*$ and $n \in \N$, the  stochastic process
   \[ \big( (B(t)u_1^*, \ldots, B(t)u_n^* ):\, t\ge 0\big)\]
  is a fractional Brownian motion with Hurst parameter $H$ in $\R^n$;
  \item[(b)] the covariance operator $Q:U^\ast\to U^{\ast\ast}$ of $B(1)$ defined by
  \begin{align*}
    \scapro{Qu^\ast}{v^\ast}=E\big[ \big(B(1)u^\ast\big)\big(B(1)v^\ast\big)\big]
    \qquad\text{for all }u^\ast, v^\ast\in U^\ast,
  \end{align*}
    is $U$-valued.
  \end{itemize}
 \end{defin}
By applying  part (a) for $n=2$ it follows that a cylindrical fBm $(B(t):\,t\ge 0)$
with covariance operator $Q$  obeys
\begin{align*}
  E\big[(B(s)u^\ast)(B(t)v^\ast)\big]=\scapro{Qu^\ast}{v^\ast}R(s,t)
\end{align*}
for all $s,t\ge 0$ and $u^\ast, v^\ast\in U^\ast$. Note that if $H=\frac{1}{2}$ then Definition \ref{definition fBm cylindrical}
 covers the cylindrical Wiener process as defined in \cite{KallianpurXiong, metivier_pellaumail80, riedle11}.

 Definition \ref{definition fBm cylindrical} involves all possible $n$-dimensional projections of the process, but since we are dealing with Gaussian processes the condition can be simplified using only two-dimensional projections.
\begin{lemma}\label{lemma cyl fBm 2 dim}
For a cylindrical process $B:=(B(t):\,t\ge 0)$ in $U$ the following are equivalent:
 \begin{itemize}
  \item[(a)]  $B$ is a cylindrical fractional Brownian motion with Hurst parameter $H\in(0,1)$;
 \item[(b)]  $B$ satisfies:
    \begin{itemize}
     \item[(i)]
    for each $u^*,v^*\in U^*$ the stochastic process
    $\big((B(t)u^*,B(t)v^*):\, t\ge 0\big)$ is a two-di\-men\-sio\-nal fBm;
     \item[(ii)] the covariance operator of $B(1)$ is $U$-valued.
    \end{itemize}
 \end{itemize}
\end{lemma}

\begin{proof}
We have to prove only the implication (b) $\Rightarrow$ (a).
For $u_1^*,\ldots,u_n^*\in U^*$ define the stochastic process
$Y=\big((B(t)u_1^*,\ldots,B(t)u_n^*):\,t\ge 0\big)$. It follows that $Y$ is
Gaussian and satisfies $E[\scapro{\alpha}{Y(t)}]=0$
for all $t\ge 0$ and $\alpha=(\alpha_1,\dots, \alpha_n)\in\R^n$
since
\begin{align*}
  \langle \alpha ,Y(t)\rangle =\sum_{i=1}^n \alpha_{i} B(t)u^*_i= B(t)\left(\sum_{i=1}^n  \alpha_{i}  u^*_i\right) .
\end{align*}
Let  $M=(m_{i,j})_{i,j=1}^n$ be the $n$-dimensional matrix defined by
\begin{align*}
  m_{i,j}= E\left[\big(B(1)u_i^*\big)\big(B(1)u_j^*\big)\right],
  \qquad i,j=1,\dots, n.
\end{align*}
Since it follows from (b) that $E\big[(B(s)u_i^*)(B(t)u_j^*)\big]=m_{i,j}R(s,t)$ for all $s,t\ge 0$ and $i,j=1,\dots, n$ we obtain
\begin{align*}
E\big[\scapro{\alpha}{Y(s)}\scapro{\beta}{Y(t)}\big]
&=E\left[\sum_{i=1}^n \sum_{j=1}^n \alpha_i\beta_j \big(B(s)u^*_i\big)  \big(B(t)u^*_j \big)\right]\\
&=\sum_{i=1}^n \sum_{j=1}^n \alpha_i\beta_j m_{i,j} R(s,t) \\
&= \scapro{M\alpha}{\beta}  R(s,t)
\end{align*}
for each $\alpha=(\alpha_1,\dots, \alpha_n)$ and $ \beta=(\beta_1,\dots, \beta_n)$ in $\R^n$.
\end{proof}


The following result provides an analogue of the Karhunen-Lo{\`e}ve expansion
for cylindrical Wiener processes.
\begin{theo}\label{theorem cylindrical fBm as series}
For a cylindrical process $B:=(B(t):\,t\ge 0)$ the following are equivalent:
 \begin{itemize}
 \item[(a)] $B$ is a cylindrical fractional Brownian motion with Hurst parameter $H\in(0,1)$;
 \item[(b)] there exist a Hilbert space $X$ with an orthonormal basis $(e_k)_{k\in\N}$, $i\in\mathcal L(X,U)$
and a sequence $(b_k)_{k\in\N}$ of independent, real valued standard fBms  with Hurst parameter $H\in (0,1)$   such that
\begin{align}\label{eq.fBMrepseries}
 B(t)u^*=\sum_{k=1}^\infty \langle ie_k,u^*\rangle  b_k(t)
\end{align}
     in $L^2_P(\Omega;\R)$ for all $u^*\in U^*$ and  $t\ge 0$.
 \end{itemize}
 In this situation the covariance operator of $B(1)$ is given by $Q=ii^\ast:U^\ast\to U$.
\end{theo}
\begin{proof}
The implication (a) $\Rightarrow$ (b) can be proved as Theorem 4.8 in \cite{appelbaum_riedle10}.
For establishing the implication (b) $\Rightarrow$ (a), it is immediate that the
right hand side of \eqref{eq.fBMrepseries} converges.  Fix $u_1^*,\ldots,u_n^* \in U^*$ and define the
 $n$-dimensional stochastic process $Y:=(Y(t):\, t\ge 0)$  by
\begin{align*}
Y(t):&=(B(t) u_1^*,\ldots,B(t) u_n^*)
\qquad\text{for all }t\ge 0.
\end{align*}
It follows that $Y$ is
Gaussian and satisfies $E[\scapro{\alpha}{Y(t)}]=0$
for all $t\ge 0$ and $\alpha=(\alpha_1,\dots, \alpha_n)\in\R^n$
since
\begin{align*}
  \langle \alpha ,Y(t)\rangle =\sum_{i=1}^n \alpha_{i} B(t)u^*_i= B(t)\left(\sum_{i=1}^n  \alpha_{i}  u^*_i\right) .
\end{align*}
Let $M=(m_{i,j})_{i,j=1}^n$ be the $n\times n$-dimensional covariance matrix of the random vector $Y(1)$, that is $m_{i,j}:=E\big[ (B(1)u_i^\ast)(B(1)u_j^\ast)\big]$. The definition of $Y$ yields
\begin{align*}
     m_{i,j}
 =  \sum_{k=1}^\infty \sum_{\ell=1}^\infty \langle ie_k,u_i^*\rangle  \langle ie_\ell,u_j^*\rangle E[  b_k(1)   b_\ell(1)]
 =  \sum_{k=1}^\infty   \langle ie_k,u_i^*\rangle \langle ie_k,u_j^*\rangle.
\end{align*}
Let $\alpha=(\alpha_1,\dots, \alpha_n)\in \R^n$ and $\beta=(\beta_1,\dots, \beta_n)\in\R^n$. By using the independence of $b_k$ and $b_\ell$ for each $k\neq \ell$ we obtain for every $s,t\ge 0$
\begin{align*}
&  E\big[\scapro{\alpha}{Y(s)}\scapro{\beta}{Y(t)}\big]\\
&\qquad =E \left[\left(\sum_{i=1}^n \alpha_i \sum_{k=1}^\infty \scapro{ie_k}{u_i^\ast}b_k(s) \right)\left(\sum_{j=1}^n \beta_j \sum_{\ell=1}^\infty \scapro{ie_\ell}{u_j^*} b_\ell(t) \right)\right]\\
&\qquad=\sum_{i=1}^n \sum_{j=1}^n \alpha_i \beta_j   \sum_{k=1}^\infty \sum_{\ell=1}^\infty  \scapro{ie_k}{u_i^*}\scapro{ie_\ell}{u_j^*} E \left[  b_k(s)  b_\ell(t) \right]\\
&\qquad=\sum_{i=1}^n \sum_{j=1}^n \alpha_i \beta_j   \sum_{k=1}^\infty \scapro{ie_k}{u_i^*}\scapro{ie_k}{u_j^*} E \left[  b_k(s)  b_k(t) \right]\\
&\qquad=\sum_{i=1}^n \sum_{j=1}^n \alpha_i \beta_j  m_{i,j} R(s,t)\\
&\qquad= \scapro{M\alpha}{\beta} R(s,t).
\end{align*}
It is left to prove that $B(1):U^\ast\to L^0_P(\Omega;\R)$ is continuous and its covariance operator $Q:U^*\to {U^{*}}'$  is    $U$-valued. By independence of $b_k$ and $b_\ell$ for $k\neq \ell$ it follows for $u^\ast\in U^\ast$ that
\begin{align*}
\phi_{B(1)}(u^*)
&= \prod_{k=1}^\infty E\big[\exp\left(\imath\, \scapro{ie_k}{u^\ast} b_k(1)\right)\big] \\
&=\prod_{k=1}^\infty  \exp\left( -\tfrac{1}{2}\scapro{ie_k}{u^*}^2 \right)
= \exp\left( -\tfrac{1}{2}  \| i^*u^*\|_{X }^2 \right).
\end{align*}
Thus, the characteristic function $\phi_{B(1)}:U^\ast\to {\mathbb C}$ is continuous, which entails the continuity of $B(1)$ by \cite[Pro.~IV.3.4]{vakhaniya_et.al.87}. Moreover, it follows that  $Q=ii^\ast$, that is the covariance operator $Q$ is $U$-valued and of the claimed form.
\end{proof}

\begin{example}\label{ex.fBMheatequation}
Let $U$ be a Hilbert space with orthonormal basis $(e_k)_{k\in\N}$, identify the dual space $U^\ast$ with $U$,
and let $(q_k)_{k\in\N}\subseteq \R$ be a sequence satisfying $\sup_{k\in\N} \abs{q_k}<\infty$.
It follows by Theorem \ref{theorem cylindrical fBm as series} that
for an arbitrary sequence $(b_k )_{k\in\N}$ of independent, real valued standard fBms,
the series
\begin{align*}
B(t) u: =\sum_{k=1}^\infty  q_k\scapro{e_k}{u} b_k(t), \qquad u\in U,
\end{align*}
defines a cylindrical fBm $(B(t):\,t\ge 0)$ in $U$. The covariance operator
$Q$ is given by $Q=ii^\ast$, where $i:U\to U$ is
defined as $i u =\sum_{k=1}^\infty q_k\scapro{e_k}{u}e_k$.
\end{example}


\begin{example} \label{example cyl fBm in L1}
For a set  $D\in{\mathcal B}(\R^n)$ let $(e_k)_{k\in \N}\subseteq L^2(D;\R)$ be an orthonormal basis and   let $(\tau_k)_{k\in\N}$ be a sequence of functions $\tau_k\in  L^2(D;\R)$ satisfying  $\sum_{k=1}^\infty \norm{\tau_k}^2_{L^2} <\infty $. Applying Cauchy-Schwarz inequality twice shows that
\begin{align}\label{eq.examplesdefi}
  i:L^2(D;\R)\to L^1(D;\R), \qquad if=\sum_{k=1}^\infty \scapro{e_k}{f} \tau_k(\cdot) e_k(\cdot)
\end{align}
defines a linear and continuous mapping.
It follows from Theorem \ref{theorem cylindrical fBm as series} that
for an arbitrary sequence $(b_k )_{k\in\N}$ of independent, real valued standard fBm,
the series
\[
B(t) f: =\sum_{k=1}^\infty  \scapro{i e_k}{f} b_k(t), \qquad f\in L^\infty(D;\R),
\]
defines a cylindrical fBm $(B(t):\,t\ge 0)$ in $L^1(D;\R)$ with covariance  operator $Q=ii^*:L^\infty(D;\R)\to L^1(D;\R)$.
\end{example}

\begin{example}\label{example cyl fBm in L1 special case}
A special case of Example \ref{example cyl fBm in L1} is obtained by choosing the functions
$\tau_k\in L^2(D;\R)$ as $\tau_k=q_k\1_{A_k}$ for $q_k\in \R$ and
$A_k\in \mathcal B (D)$ satisfying  $\sum_{k=1}^\infty q_k^2 \text{ \rm Leb}(A_k)<\infty$. Then the cylindrical fBm of Example  \ref{example cyl fBm in L1} has the form
\[
B(t)f = \sum_{k=1}^\infty q_k \langle \1_{A_k} e_k, f \rangle b_k(t).
\]
This process can be considered as an anisotropic cylindrical  fractional Brownian sheet in $L^1(D;\R)$
since its covariance structure might vary in different directions.
\end{example}

In the final part of this section we consider the relation between  cylindrical and genuine fractional Brownian motion in a separable Banach space $U$. For this purpose, we generalise the definition of a fractional Brownian motion in $\R^n$ to  Banach spaces. This definition is consistent with others in the literature, in particular the one in \cite{duncan_et.al.06} for Hilbert spaces.
\begin{defin}\label{definition fBm honest}
 A $U$-valued Gaussian stochastic process $(Y(t):\, t\ge 0)$ is called a {\em fractional Brownian motion in $U$
 with Hurst parameter $H\in (0,1)$} if there exists a mapping $Q:U^*\to U $ such that
\begin{align*}
  \langle Y(t), u^*\rangle =0,\qquad\quad
   E\big[\scapro{Y(s)}{u^*}\scapro{Y(t)}{v^*}\big] = \scapro{Qu^*}{v^*} R(s,t)
\end{align*}
for all $s,t\ge 0$ and $u^\ast,v^\ast\in U^\ast$.
\end{defin}
By taking $s=t=1$ it follows that
\begin{align*}
  \scapro{Qu^*}{v^*} = E\big[\scapro{Y(1)}{u^*}\scapro{Y(1)}{v^*}\big]
  \qquad\text{for all }u^\ast, v^\ast\in U^\ast.
\end{align*}
Thus, $Q$ is the covariance operator of the Gaussian measure $P_{Y(1)}$ and it must be a symmetric and positive
operator in $\L(U^\ast, U)$.


Clearly every fBm in a Banach space $U$ is a cylindrical fBm in $U$ and thus, it obeys the representation \eqref{eq.fBMrepseries}.
However, the operator $i$, or in other words the embedding of the reproducing kernel Hilbert space, must yield a Radon measure in
$U$, which basically leads to the following result:
\begin{theo}\label{theorem honest fBm as series}
For a $U$-valued stochastic process $Y:=(Y(t):\, t\ge 0)$ the following are equivalent:
\begin{itemize}
  \item[(a)] $Y$ is a fBm  in $U$ with Hurst parameter $H\in (0,1)$;
  \item[(b)] there exist a Hilbert space $X$ with an orthonormal basis $(e_k)_{k\in \N}$, a $\gamma$-radonifying operator $i\in \mathcal L(X,U)$ and independent, real valued standard fBms $(b_k)_{k\in \N}$  such that
      \[ Y(t)= \sum_{k=1}^\infty ie_k\,  b_k(t) \]
      in $ L^2_P(\Omega; U )$ for all $t\ge 0$.
 \end{itemize}
In this situation the covariance operator of $Y(1)$ is given by $Q=ii^\ast:U^\ast\to U$.
\end{theo}
\begin{proof}
  The result can be proved as Theorem 23 in \cite{riedle11}.
\end{proof}

In the literature a fractional Brownian motion in a Hilbert space is often defined by a series representation as in
Theorem \ref{theorem honest fBm as series}, in which case  the space of $\gamma$-radonifying operators coincides with Hilbert-Schmidt operators.

If  $(B(t):\, t\ge 0)$ is a  cylindrical fBm which is induced by a  $U$-valued process $(Y(t), t\ge 0)$, i.e.
\begin{align}\label{equation honest fBm}
B(t)u^*=\scapro{Y(t)}{u^*}
\qquad\text{for all }t\ge 0, \, u^\ast\in U^\ast,
\end{align}
then $Y$ is a $U$-valued fBm. Vice versa, if $Y$ is a $U$-valued fBm then $B$ defined by \eqref{equation honest fBm} is a cylindrical fBm, and in both cases the covariance operators coincide.
This can be seen by the fact that \eqref{equation honest fBm} determines uniquely the characteristic functions of $\big(B(s),B(t)\big)$ and $\big(Y(s),Y(t)\big)$ for all $s,t\ge 0$. Moreover, a cylindrical fBm with the representation \eqref{eq.fBMrepseries} is a $U$-valued fBm if and only if the embedding $i$ is $\gamma$-radonifying. This result can be established as in  \cite[Th.25]{riedle11}.

\begin{example} \label{example class fBm in L2}
If we assume in Example \ref{example cyl fBm in L1} that the functions $\tau_k$ are in $L^\infty(D;\R)$ and satisfy  $\sum \norm{\tau_k}_{\infty}<\infty$ then
the mapping $i$, defined in \eqref{eq.examplesdefi}, maps to $L^2(D;\R)$. Moreover,
$i$ is a Hilbert-Schmidt operator, as
\[
\sum_{k=1}^\infty  \| i e_k\|^2_{L^2}   =  \sum_{k=1}^\infty\| \tau_k e_k \|^2_{L^2} \le
 \sum_{k=1}^\infty \|\tau_k \|^2_{\infty}.
\]
Since $\gamma$-radonifying and Hilbert-Schmidt operators coincide in Hilbert spaces,
Theorem \ref{theorem honest fBm as series} implies that the cylindrical fBm in Example
\ref{example class fBm in L2} is induced by a genuine fractional Brownian motion in
$L^2(D;\R)$.
\end{example}


\section{Integration}\label{sc: integration}

In this section we introduce the stochastic integral
$\int \Psi(s)\,\ud B(s)$ as a $V$-valued random variable for  deterministic, operator valued functions $\Psi:[0,T]\to \L(U,V)$, where $V$ is another separable Banach space.
Our approach is based on the idea to introduce firstly a cylindrical random variable $Z_{\Psi}:V^\ast\to L^0_P(\Omega;\R)$ as a {\em cylindrical integral}. Then we
 call a $V$-valued random variable $I_\Psi:\Omega\to V$
 the {\em stochastic integral of $\Psi$} if it satisfies
\begin{align*}
  Z_\Psi v^\ast= \scapro{I_\Psi}{v^\ast}\qquad\text{for all }
  v^\ast\in V^\ast.
\end{align*}
In this way, the stochastic integral $I_\Psi$ can be considered as a {\em stochastic Pettis integral}.
This approach enables us to have a candidate of the stochastic integral, i.e. the cylindrical random variable $Z_\Psi$,  under very mild conditions at hand because cylindrical random variables are  more general objects than genuine random variables. The final requirement, that the cylindrical random variable $Z_\Psi$ is in fact a classical Radon random variable, can be equivalently described in terms of the corresponding covariance operator and thus, it solely depends on geometric properties of the underlying Banach space $V$.

For defining the cylindrical integral, recall the representation of a cylindrical fBm $(B(t):\,t\ge 0)$ with Hurst parameter $H\in (0,1)$ in the Banach space $U$, according to Theorem \ref{theorem cylindrical fBm as series}:
\begin{equation}\label{equation represent fBm}
B(t)u^*=\sum_{k=1}^\infty \langle i e_k, u^*\rangle  b_k(t)\qquad\text{for all }u^\ast\in U^\ast,
\,t\ge 0.
\end{equation}
Here, $X$ is a Hilbert space with an orthonormal basis $(e_k)_{k\in\N}$, $i:X\to U$ is a linear, continuous mapping and $(b_k)_{k\in\N}$ is a sequence of independent, real valued standard fBms. If we assume momentarily that we have already introduced a stochastic integral $\int_0^T \Psi(t)\,\ud B(t)$ as a $V$-valued random variable, then the representation \eqref{equation represent fBm} of $B$ naturally results in
\begin{align}\label{eq.motivcylint}
\sum_{k=1}^\infty \int_0^T\langle \Psi(t)ie_k, v^*\rangle\, \ud b_k(t)\qquad\text{for all }v^\ast\in V^\ast.
\end{align}
By swapping the terms in the dual pairing, the integrals can be considered as the Fourier coefficients of the $X$-valued integral
\begin{align*}
  \int_0^T i^\ast \Psi^\ast(t) v^\ast\, \ud b_k(t),
\end{align*}
which we introduce in Section \ref{sc: wiener integrals for hilbert space valued fcs}. This results in
the minimal requirement that  the function $t\mapsto i^\ast\Psi^\ast(t)v^\ast $ must be integrable with respect to the real valued standard fBm $b_k$ for every $v^\ast\in V^\ast$ and $k\in\N$, that is the function $\Psi$ must be in the linear space
\[
\mathcal I:=\{ \Phi:[0,T]\to \mathcal L(U,V):\,  i^* \Phi^*(\cdot)v^* \in\widehat{ \mathcal M} \text{ for all } v^*\in V^* \:\}.
\]
Here, $\M=\M_X$ denotes the Banach space of functions $f:[0,T]\to X$ introduced in Section \ref{sc: wiener integrals for hilbert space valued fcs}.
For this class of integrands we have the following property.
\begin{prop}\label{prop continuity of T:V* to M}
For each $\Psi\in {\mathcal I}$ the mapping
\begin{align*}
  L_\Psi:V^\ast\to \M, \qquad L_\Psi v^\ast=i^\ast\Psi^\ast(\cdot)v^\ast
\end{align*}
is linear and  continuous.
\end{prop}

\begin{proof}
The operator $L=L_\Psi$ is linear and takes values in $\M$  by definition of $\mathcal I$. We prove that $L$ is continuous by the closed mapping theorem. For this purpose, let $v_n^*\to v_0^*$ in $V^*$ and $Lv_n^*\to g\in \widehat{\mathcal M}$. We consider the cases $H<1/2$ and $H>1/2$ separately.

\emph{Case $H<1/2$.} From the Hardy-Littlewood inequality in weighted spaces, see \eqref{eq.Hardy-Littlewood-weighted}, it follows that
\begin{align*}
  \norm{f}_{L^q}\le c (b_H \Gamma(H+\tfrac{1}{2}))^{-1/2}\norm{f}_{\mathcal M}\qquad
  \text{for all }f\in {\mathcal M}
\end{align*}
for a constant $c>0$ and $q=\tfrac{1}{H}$.
Consequently, the convergence $Lv_n^\ast \to g$ in $\mathcal M$ implies that there exists  a subsequence $(n_k)_{k\in\N}\subseteq \N$ such
that $Lv_{n_k}^*(t)\to g(t)$ as $k\to\infty$  for Lebesgue almost all $t\in [0,T] $. On the other hand, we have $i^* \Psi^*(t)v_{n_k}^*\to i^* \Psi^*(t)v_0^* $ in $X$  as $k\to\infty$ for all $t\in [0,T]$, because $i^*$ and $\Psi^*(t)$ are continuous. Consequently, we arrive at
$g(t)= i^* \Psi^*(t)v_0^*$ for Lebesgue almost all $t\in [0,T]$,
and thus, $g=Lv_0^\ast$ as functions in $L^2([0,T];X)$.

\emph{Case $H>1/2$.} In this case $\widehat{\mathcal M}= | \mathcal M |$. Let us remark, that if $f\in | \mathcal M 	 |$ then $f\in L^1([0,T];X)$ and
\begin{align*}
(2T)^{2H-2} \|f \|^2_{L^1} &=  (2T)^{2H-2} \int_0^T \int_0^T \|f(s)\|  \|f(t)\| \,\ud s \, \ud t \\
&\le  \int_0^T \int_0^T \|f(s)\|  \|f(t)\| |s-t|^{2H-2} \,\ud s\, \ud t \\
&=\frac{1}{H(2H-1)}\|f\|_{|\mathcal M|}^2.
\end{align*}
Using this fact, the convergence $Lv_n^*\to g$ in $|\mathcal M|$ implies that $Lv_{n_k}^\ast (t)\to g(t)$ as $k\to\infty$
for Lebesgue almost all $t\in [0,T]$  for a subsequence $(n_k)_{k\in\N}\subseteq \N$. The continuity of the mapping $v^\ast\mapsto i^\ast\Psi^\ast(t)v^\ast$ for all $t\in [0,T]$
shows that $g(t)=i^\ast\Psi^\ast(t)v^\ast_0$ for Lebesgue
almost all $t\in [0,T]$ and thus, $g=Lv_0^\ast$ in $|\mathcal M|$.
\end{proof}

Before we establish the existence of the cylindrical integral as motivated in \eqref{eq.motivcylint}, we introduce an operator which will turn out to be the factorisation of the covariance operator of the cylindrical integral.
\begin{lemma}\label{le.cov}
  For every $\Psi\in \mathcal I$ we define
  \begin{align*}
 \scapro{\Gamma_{\Psi}f}{v^*}= \int_0^T [K^*( i^* \Psi^*(\cdot)v^*)(t), f(t)]\,\ud t\quad\text{for all }f\in L^2([0,T];X),\, v^\ast\in V^\ast.
\end{align*}
In this way, one obtains a linear, bounded operator $\Gamma_{\Psi}:L^2([0,T];X)\to V^{**}$.
\end{lemma}
\begin{proof}
Proposition \ref{prop continuity of T:V* to M}, together with equation \eqref{eq.KisoMmod}, implies
\begin{align*}
\abs{ \scapro{\Gamma_{\Psi}f}{v^*}}
&= \abs{\scapro{K^\ast(i^\ast\Psi^\ast(\cdot)v^\ast)}{f}_{L^2}  } \\
& \le  c_1\norm{  i^* \Psi^*(\cdot)v^\ast}_{\M}  \norm{f}_{L^2}
 \le c_2 \norm{  v^\ast}_{V^*}  \norm{f}_{L^2} ,
\end{align*}
for some constants $c_1,\,c_2>0$, which shows boundedness of $\Gamma_{\Psi}$.
\end{proof}

\begin{prop}\label{prop integral well def and isometry}
Let the fBm $B$ be represented in the form  \eqref{equation represent fBm}.
Then for each $\Psi\in \mathcal I$ the mapping
\begin{equation}\label{equation def cylindrical integral}
Z_\Psi:V^\ast\to L^2_P(\Omega;\R),\qquad
Z_\Psi v^*:=\sum_{k=1}^\infty \int_0^T\langle \Psi(t)ie_k, v^*\rangle\, \ud b_k(t)
\end{equation}
defines a Gaussian cylindrical random variable in $V$ with covariance operator
$ Q_{\Psi}:V^\ast\to V^{**}$, factorised by $ Q_{\Psi}  =  \Gamma_{\Psi}\Gamma_{\Psi}^\ast$.
Furthermore, the cylindrical random variable $Z_\Psi$ is  independent of
the representation \eqref{equation represent fBm}.
\end{prop}
\begin{proof}
Since  $\langle \Psi(\cdot)ie_k, v^*\rangle =[e_k,i^*\Psi^*(\cdot)v^* ]$ and $i^*\Psi^*(\cdot)v^* \in \widehat{ \mathcal M}$ for every $v^\ast\in V^\ast$, Proposition \ref{prop link operator K and K} guarantees that the one-dimensional integrals in \eqref{equation def cylindrical integral}  are well defined, and it implies that
\begin{align*}
\|Z_\Psi v^*\|_{L^2_P}^2
&=  \sum_{k=1}^\infty  E \left\vert \int_0^T  \langle \Psi(t)ie_k, v^*\rangle \ud b_k(t) \right\vert^2  \\
&=  \sum_{k=1}^\infty  \int_0^T \left\vert K_{\R}^* \big( \langle \Psi(\cdot)ie_k, v^*\rangle\big) (t) \right\vert^2 \ud t \\
&=  \sum_{k=1}^\infty  \int_0^T \left\vert  K_{\R}^* \big( [ e_k, i^*\Psi^*(\cdot)v^*]\big) (t) \right\vert^2 \ud t\\
&=  \sum_{k=1}^\infty  \int_0^T \left\vert [ e_k, K_X^*\big( i^*\Psi^*(\cdot)v^*\big) (t)] \right\vert^2 \ud t \\
&=  \sum_{k=1}^\infty  \int_0^T  \left[e_k ,\big(\Gamma_{\Psi}^\ast v^\ast \big)(t)\right]^2 \, \ud t \\
&=\|\Gamma_{\Psi}^\ast v^*\|^2_{  L^2}.
\end{align*}
Consequently, the sum in \eqref{equation def cylindrical integral} converges in $L^2_P(\Omega;\R)$ and the limit is a zero mean Gaussian random variable. The continuity of the operator $\Gamma_{\Psi}^\ast: V^\ast \to L^2([0,T];X)$ implies  the continuity
of $Z_\Psi:V^\ast\to L^2_P(\Omega;\R)$. It follows for the characteristic function of $Z_\Psi$ that
\begin{align*}
  \phi_{Z_\Psi}(v^\ast)=\exp\left(-\tfrac{1}{2}\norm{\Gamma_\Psi^\ast v^\ast}^2_{L^2}\right)
 \qquad\text{for all }v^\ast\in V^\ast.
\end{align*}
Since Lemma \ref{le.cov} implies that
\begin{align*}
  \norm{\Gamma_\Psi^\ast v^\ast}_{L^2}^2=
  \scapro{\Gamma_\Psi \Gamma_\Psi^\ast v^\ast}{v^\ast}
 \qquad\text{for all }v^\ast\in V^\ast ,
\end{align*}
it follows that the covariance operator $Q_\Psi$ of $Z_\Psi$ obeys
 $Q_\Psi=\Gamma_\Psi \Gamma_\Psi^\ast$.


The independence of $Z_\Psi$ of the representation \eqref{equation represent fBm}  can be established as in \cite[Le.2]{riedle11}.
\end{proof}

For $\Psi\in {\mathcal I}$ we call the cylindrical random variable $Z_\Psi$, defined in \eqref{equation def cylindrical integral}, the {\em cylindrical integral of $\Psi$}.
Apart from the restriction of the space $\mathcal M$ of all integrable distributions to $\M$, the condition for  a mapping $\Psi$ to be in $\mathcal I$  is the minimal requirement to guarantee  that the real valued integrals in \eqref{eq.motivcylint} exist. Thus without any further condition  the cylindrical integral $Z_\Psi$ exists in the Banach space $U$. However, in order to obtain
that the cylindrical integral $Z_\Psi$  extends to a genuine random variable in $U$,
the integrand must exhibit further properties.

\begin{defin}
A function $\Psi\in {\mathcal I}$ is called {\em stochastically integrable}
if there exists a random variable $I_\Psi:\Omega\to V$ such that
\begin{align*}
  Z_\Psi v^\ast=\scapro{I_\Psi}{v^\ast}
  \qquad\text{for all }v^\ast\in V^\ast,
\end{align*}
where $Z_\Psi$ denotes the cylindrical integral of $\Psi$. We use the notation
\begin{align*}
  I_\Psi:=\int_0^T \Psi(t)\,\ud B(t).
\end{align*}
\end{defin}

In other words, a function $\Psi\in {\mathcal I}$ is stochastically integrable if and only if
the cylindrical random variable $Z_\Psi$ is induced by a Radon random variable. This occurs if and only if the cylindrical distribution of $Z_\Psi$ extends to a Radon measure. In Sazonov spaces this is equivalent to the condition that the characteristic function of $Z_\Psi$ is Sazonov continuous. However, since the cylindrical distribution of $Z_\Psi$ is Gaussian, one can equivalently express the stochastic integrability in terms of the covariance operator.
\begin{theo}\label{theo I_t honest process}
For $\Psi\in \mathcal I$ the following are equivalent:
\begin{itemize}
  \item[(a)] $\Psi$ is stochastically integrable;
  \item[(b)]  the operator $\Gamma_{\Psi}$  is $V$-valued and  $\gamma$-radonifying.
\end{itemize}
\end{theo}

\begin{proof}
(b) $\Rightarrow$ (a).
Let $\gamma$ be the canonical Gaussian cylindrical measure on $L^2([0,T];X)$.
It follows from Proposition \ref{prop integral well def and isometry}
that the cylindrical distribution of $Z_\Psi$ is the image cylindrical measure $\gamma\circ \Gamma_\Psi^{-1}$. According to \cite[Thm.IV.2.5, p.216]{vakhaniya_et.al.87}, the cylindrical random variable $Z_\Psi$ is induced by a $V$-valued random variable if and only if its cylindrical distribution $\gamma\circ \Gamma_\Psi^{-1}$ extends to a Radon measure on ${\mathcal B}(V)$, which is guaranteed
by (b).

(a) $\Rightarrow$ (b). The proof follows closely some arguments in the proof of Theorem 2.3 in \cite{JanLutz2005}.
Let $Q:V^\ast\to V$ be the covariance operator of the Gaussian random variable $\int \Psi(t)\,\ud B(t)$. Proposition  \ref{prop integral well def and isometry} implies that $Q=\Gamma_\Psi \Gamma_\Psi^\ast:V^\ast\to V$.
Define the set $S:=\{K^\ast(i^\ast\Psi^\ast(\cdot)v^\ast):\,v^\ast\in V^\ast\}$, which is a subset of $L^2([0,T];X)$. By the very definition of $\Gamma_{\Psi}$, a function $f\in L^2([0,T];X)$ is in ker$\,\Gamma_{\Psi}$
if and only if $f\perp S$, which yields
\begin{align}\label{eq.decompL^2}
  L^2([0,T];X)=\bar{S}\oplus \text{\rm ker}\,\Gamma_{\Psi}.
\end{align}
Since for all $v^\ast, w^\ast\in V^\ast$ we have
\begin{align*}
  \scapro{\Gamma_{\Psi}K^\ast(i^\ast\Psi^\ast(\cdot)v^\ast)}{w^\ast}
  =\scapro{\Gamma_{\Psi}\Gamma_{\Psi}^\ast v^\ast}{w^\ast}
  =\scapro{Q v^\ast}{w^\ast},
\end{align*}
it follows that $\Gamma_{\Psi}K^\ast(i^\ast\Psi^\ast(\cdot)v^\ast)=Qv^\ast$ for
all $v^\ast\in V^\ast$. Consequently,
$\Gamma_{\Psi}f\in V$ for all $f\in S$ and the decomposition \eqref{eq.decompL^2} implies
that $\Gamma_{\Psi}f\in V$ for all $f\in L^2([0,T];X)$. Clearly, since $Q$ is a Gaussian covariance operator, the operator $\Gamma_{\Psi}$ is $\gamma$-radonifying, which completes the proof.
 \end{proof}

\begin{coroll}\label{cor.domintegrable}
If $\Phi$ and $\Psi$ are mappings in $\mathcal I$ satisfying
\begin{align*}
  \norm{i^\ast \Phi^\ast(\cdot)v^\ast}_{\mathcal M}\le c
   \norm{i^\ast \Psi^\ast(\cdot)v^\ast}_{\mathcal M}\qquad\text{for all }
     v^\ast\in V^\ast,
\end{align*}
for a constant $c>0$
and if $\Psi$ is stochastically integrable then $\Phi$ is also
stochastically integrable.
\end{coroll}
\begin{proof} The proof follows some arguments in the proof of Theorem 2.3 in \cite{JanLutz2005}.
Define  the operator $Q:=\Gamma_{\Phi}\Gamma_{\Phi}^\ast:V^\ast\to V^{\ast\ast}$.
The isometry \eqref{equation isometry wiener integral Hilbert} implies for every $v^\ast\in V^\ast$ that
   \begin{align*}
     \scapro{v^\ast}{Qv^\ast}^2
     = \scapro{\Gamma_{\Phi}^\ast v^\ast}{\Gamma_{\Phi}^\ast v^\ast}^2_{L^2}
     &= \norm{K^\ast\big(i^\ast \Phi^\ast(\cdot)v^\ast\big)}^2_{L^2}\\
     &= \norm{i^\ast  \Phi^\ast(\cdot)v^\ast}^2_{\mathcal M}\\
     &\le c \norm{i^\ast \Psi^\ast(\cdot)  v^\ast}^2_{\mathcal M}
     = c \scapro{\Gamma_{\Psi}\Gamma_{\Psi}^\ast v^\ast}{v^\ast}.
   \end{align*}
Since $\Gamma_{\Psi}\Gamma_{\Psi}^\ast$ and $Q$ are positive, symmetric operators in $\L(V^\ast, V^{\ast\ast})$ and the first one is $V$-valued according to Theorem \ref{theo I_t honest process}, it follows by an argument based on a result of the domination of Gaussian measures, see \cite[Sec.1.1]{JanLutz2005}, that $Q$ is also $V$-valued and a Gaussian covariance operator. As in the proof of the implication (a) $\Rightarrow$ (b) in Theorem \ref{theo I_t honest process} we can conclude that $\Gamma_\Phi$ is $V$-valued.
\end{proof}


If the mapping $\Psi\in \mathcal I$ is stochastically integrable, Proposition \ref{pro.extrefinM}
implies for each $t\in [0,T]$ that $\1_{[0,t]}\Psi\in \mathcal I$  and it satisfies
\begin{align*}
  \norm{\1_{[0,t]}i^\ast \Psi^\ast(\cdot)v^\ast}_{\mathcal M}\le
  c_t \norm{i^\ast \Psi^\ast(\cdot)v^\ast}_{\mathcal M}
   \qquad\text{for all }
     v^\ast\in V^\ast,
\end{align*}
for a constant $c_t>0$. Corollary \ref{cor.domintegrable} enables us to conclude
that $\1_{[0,t]}\Psi$ is stochastically integrable, and thus we can define
\begin{align*}
  \int_0^t \Psi(s)\,\ud B(s):=\int_0^T \1_{[0,t]}(s)\Psi(s)\,\ud B(s)
  \qquad\text{for all }t\in [0,T].
\end{align*}
The integral process $\big(\int_0^t \Psi(s)\,\ud B(s):\, t\in [0,T]\big)$ is continuous in $p$-th mean for each $p\ge 1$.
In order to see that let $t_n\to t$ as $n\to\infty$ for $t_n\ge t$ and let $Q^{(n)}_{\Psi}$ denote the covariance
operator of the Gaussian random variable $\int_t^{t_n} \Psi(s)\,\ud B(s)$. It follows for each $v^\ast\in V^\ast$
that
\begin{align*}
  \scapro{Q_{\Psi}^{(n)}v^\ast}{v^\ast}
  = \norm{K^\ast(\1_{[t,t_n]}(\cdot)i^\ast\Psi^\ast(\cdot)v^\ast)}^2_{L^2}
  = \norm{\1_{[t,t_n]}(\cdot)i^\ast\Psi^\ast(\cdot)v^\ast}^2_{\mathcal M}.
\end{align*}
Each $f\in \M$ satisfies $\norm{\1_{[t,t_n]}(\cdot)f}_{\mathcal M}\to 0$ as $t_n\to t$ which follows from
\eqref{eq.scalarMH>0.5} in case $H>\tfrac{1}{2}$ and from results in \cite[Ch.13.3]{samko_et.al.93}
in case $H<\tfrac{1}{2}$, see also Proposition \ref{pro.extrefinM}. Consequently, we obtain
that $\scapro{Q^{(n)}_{\Psi}v^\ast}{v^\ast}\to 0$ as $t_n\to t$ and we can conclude as in the proof
of Corollary 2.8 in \cite{JanLutz2005} that the integral process is continuous in $p$-th mean.

\section{The Cauchy problem}\label{sc: cauchy problem}

In this section, we apply our previous results to consider stochastic evolution equations  driven by cylindrical fractional Brownian motions of the form
\begin{align}\label{equation stoch Cauchy problem}
\begin{split}
  \ud Y(t)&= AY(t)\,\ud t +C\,\ud B(t),\quad  t\in(0,T],\\
Y(0)&=y_0.
\end{split}
\end{align}
Here $B$ is a cylindrical fBm in a separable Banach space $U$, $A$ is a generator
 of a strongly continuous semigroup $(S(t), t\ge 0)$ in a separable Banach space $V$
and $C$ is an operator in $\L(U,V)$. The initial condition $y_0$ is an element in $V$.

The paths of a solution exhibit some kind of  regularity, which is weaker than $P$-a.s. Bochner integrable paths:
\begin{defin}
  A $V$-valued stochastic process $(X(t):\, t\in [0,T])$ is called {\em weakly Bochner regular} if for every sequence $(H_n)_{n\in\N}$ of continuous functions $H_n:[0,T]\to V^\ast$ it satisfies:
  \begin{align*}
    \sup_{t\in [0,T]} \norm{H_n(t)}\to 0
    \;\Rightarrow \; \int_0^T \abs{\scapro{X(t)}{H_{n_k}(t)}}^2\, dt \to 0
    \qquad\text{$P$-a.s. for } k\to\infty,
  \end{align*}
for a subsequence $(H_{n_k})_{k\in\N}$ of $(H_n)_{n\in\N}$.
\end{defin}


\begin{defin}
A  stochastic process $(Y(t):\, t\in[0,T])$ in $V$ is called a \emph{weak solution of \eqref{equation stoch Cauchy problem}} if it is weakly Bochner regular and for every $v^\ast \in \D(A^\ast)$ and $t\in [0,T]$
we have $P$-a.s.,
\begin{equation}\label{eq Cauchy pb}
\scapro{Y(t)}{v^*}= \scapro{y_0}{v^\ast} + \int_0^t \scapro{Y(s)}{A^\ast v^*} \, \ud s+ B(t)(C^\ast v^*).
\end{equation}
\end{defin}

From a proper integration theory we can expect that if the convoluted semigroup
$S(t-\cdot)C$ is integrable for all $t\in [0,T]$ then a weak solution of \eqref{equation stoch Cauchy problem} exists and can be represented by the usual variation of constants formula.

\begin{theo}\label{th.existencesolution}
Assume that $S(\cdot)C$ is in $\mathcal I$. Then the following are equivalent:
\begin{enumerate}
 \item[{\rm (a)}] the Cauchy problem \eqref{equation stoch Cauchy problem} has a weak solution $Y$;
  \item[{\rm (b)}] the mapping $S(\cdot)C$ is stochastically integrable.
\end{enumerate}
In this situation the solution $(Y(t):\,t\in [0,T])$ can be represented by
\begin{equation}\label{eq.varofcons}
Y(t) = S(t) y_0+ \int_0^t S(t-s) C\, \ud B(s)\qquad\text{for all $t\in[ 0,T] $.}
\end{equation}
\end{theo}

\begin{proof}
(b) $\Rightarrow$ (a):
Proposition \ref{pro.extrefinM} guarantees for each $t\in [0,T]$ that the mapping $\1_{[0,t]}  S(t-\cdot)C$ is in $\mathcal I$ and that there exists a constant $c_t>0$ such that
\begin{align*}
  \norm{\1_{[0,t]}i^\ast C^\ast S^\ast(t-\cdot)v^\ast}_{\mathcal M}\le
   c_t\norm{i^\ast C^\ast S^\ast(\cdot)v^\ast}_{\mathcal M}\qquad\text{for all }
     v^\ast\in V^\ast.
\end{align*}
 Thus Corollary \ref{cor.domintegrable} guarantees that
$\1_{[0,t]}  S(t-\cdot)C$ is stochastically integrable, which enables to define the
 stochastic integral
\begin{align*}
  X(t):= \int_0^t S(t-s) C \,\ud B(s)\qquad\text{for all }t\in [0,T].
\end{align*}
It follows from representation \eqref{equation def cylindrical integral} that
the real valued stochastic process $(\scapro{X(t)}{v^\ast}:\,t\in [0,T])$ is
adapted for each $v^\ast\in V^\ast$. Pettis' measurability theorem implies that
$X:=(X(t):\,t\in [0,T])$ is adapted.

By linearity we can assume that $y_0=0$.
The stochastic Fubini theorem for real valued fBm implies for each $v^*\in \D(A^\ast)$ and $t\in [0,T]$, that
\begin{align*}
 \int_0^t \scapro{X(s)}{A^\ast v^*}\, \ud s
 &=  \sum_{k=1}^\infty  \int_0^t  \int_0^s \langle S(s-r)C i e_k, A^*v^* \rangle \,\ud  b_k(r)\,\ud s \nonumber\\
 &= \sum_{k=1}^\infty \int_0^t   \int_r^t \langle S(s-r)C i e_k, A^*v^* \rangle \,\ud s\,\ud  b_k(r)\nonumber\\
 &= \sum_{k=1}^\infty \int_0^t  \langle S(t-r)C i e_k - C i e_k  ,v^* \rangle \,\ud b_k(r)\nonumber\\
 &=  \scapro{ \int_0^t    S(t-r)C  \,\ud B(r) }{ v^*} -\sum_{k=1}^\infty    \langle i e_k  ,C^* v^* \rangle b_k(t)\nonumber \\
 &= \scapro{X(t)}{ v^*} -    B(t) (C^*v^*),
 \end{align*}
which shows that the process $X$ satisfies \eqref{eq Cauchy pb}. In order to show that $X$ is weakly Bochner regular define $\Psi_t:=\1_{[0,t]}(\cdot) S(t-\cdot)C$ for each $t\in [0,T]$. Note that Proposition \ref{pro.extrefinM} guarantees that there exists a constant $c_t>0$ such that
\begin{align*}
\norm{\Gamma_{\Psi_t}^\ast v^\ast}_{L^2}
=\norm{\1_{[0,t]}(\cdot) i^\ast C^\ast S^\ast(t-\cdot)v^\ast}_{\mathcal M}
\le  c_t \norm{i^\ast C^\ast S^\ast(\cdot)v^\ast}_{\mathcal M}
\end{align*}
for every $v^\ast\in V^\ast$. Since the derivation of the constant $c_t$ in \cite[Ch.13.3]{samko_et.al.93} shows that $\sup_{t\in [0,T]} c_t<\infty$, the uniform boundedness principle implies that
$\sup_{t\in [0,T]}\norm{\Gamma_{\Psi_t}^\ast}_{V^\ast\to L^2}<\infty$.
Thus for a sequence $(H_n)_{n\in\N}$ of continuous mappings
$H_n:[0,T]\to V^\ast$ we obtain
\begin{align*}
E\left[\abs{\int_0^T \scapro{X(t)}{H_n(t)}\,\ud t}^2\right]
&\le T\int_0^T E\left[\abs{Z_{\Psi_t} H_n(t)}^2\right]\,\ud t \\
&= T\int_0^T \norm{\Gamma_{\Psi_t}^\ast H_n(t)}^2\,\ud t\\
&\le T^2 \sup_{t\in [0,T]}\norm{\Gamma_{\Psi_t}^\ast}^2_{V^\ast\to L^2}
 \sup_{t\in [0,T]}\norm{H_n(t)}^2,
\end{align*}
which shows the weak Bochner regularity.

(a) $\Rightarrow$ (b): by applying It\^{o}'s formula for real valued fBm, see e.g.~\cite[Thm 6.3.1]{biagini_et.al.08}, one deduces
for every continuously differentiable function $f:[0,T]\to \R$ and real valued fBm $b$
\begin{equation}\label{eq: integr by parts fBm}
\int_0^{T} f'(s) b(s) \,\ud s = f(T)b(T) - \int_0^{T} f(s) \, \ud b(s)
\qquad\text{$P$-a.s.,}
\end{equation}
where the integral on the right-hand side can be understood as a Wiener integral, since $f$ is deterministic.
Let $Y$ be a weak solution of \eqref{equation stoch Cauchy problem} and denote by $A^{\odot}$
the part of $A^\ast$ in $\overline{\D(A^\ast)}$. Then $\D(A^{\odot})$ is a weak$^\ast$-sequentially
dense subspace of $V^\ast$.  From the integration by parts formula
\eqref{eq: integr by parts fBm} it follows as in the proof of Theorem 7.1 in \cite{JanLutz2005}
that
\begin{align}\label{eq.YandconvInt-dense}
 \scapro{Y(T)}{v^*}
=  Z_\Psi v^*
\qquad\text{for all }v^\ast\in \D(A^{\odot}),
\end{align}
where $Z_\Psi$ denotes the cylindrical integral of $\Psi:=S(T-\cdot)C$.
It remains to show that \eqref{eq.YandconvInt-dense} holds for all $v^\ast\in V^\ast$, for which we mainly follow  the arguments of the proof of Theorem 2.3 in \cite{JanLutz2005}.  Observe that the random variable $Y(T)$ is Gaussian since the right hand side in \eqref{eq.YandconvInt-dense} is Gaussian for each $v^\ast\in \D(A^{\odot})$ and Gaussian distributions are closed under weak limits. Let  $R:V^\ast\to V$ and $Q:V^\ast\to V^{\ast\ast}$ denote the covariance operators of $Y(T)$ and   $Z_{\Psi}$, respectively. Since $R$ is the covariance operator of a Gaussian measure there exists
a Hilbert space $H$ which is continuously embedded by a $\gamma$-radonifying mapping $j:H\to V$ such that $R=jj^\ast$. Equality \eqref{eq.YandconvInt-dense} implies
\begin{align}\label{eq.RandQdense}
  \scapro{R v^*}{v^*} = \scapro{Q v^*}{v^*}
  \qquad\text{for all }v^\ast\in \D(A^{\odot}).
\end{align}
Let  $(v_n^*)_{n\in\N}$ be a sequence in $\D(A^{\odot})$ converging weakly$^\ast$ to $v^*$ in $V^*$. Thus $\lim_{n\to\infty}j^*v_n^*=j^*v^*$ weakly in $H$ since $H$ is a Hilbert space and $j^*$ is weak$^*$ continuous. As a consequence of the Hahn-Banach theorem one can construct a convex combination $w_n^*$ of the $v^*_n$ such that $\lim_{n\to\infty}j^*w_n^*=j^*v^*$ strongly in $H$ and $\lim_{n\to\infty}w_n^*=v^*$ weakly$^*$ in $V^*$. Since
$w_m^\ast-w_n^\ast$ is in $\D(A^{\odot})$ for all $m,n\in\N$, inequality \eqref{eq.RandQdense} implies
\begin{align*}
\norm{i^*C^*S^*(T-\cdot)(w_m^*-w_n^*)}_{\mathcal M}^2 &= \norm{K^* \left( i^*C^*S^*(T-\cdot)(w_m^*-w_n^*)\right)}_{L^2}^2\\
&=  \scapro{Q (w_m^*-w_n^*)}{w_m^*-w_n^*} \\
&=  \scapro{R(w_m^*-w_n^*)}{w_m^*-w_n^*} \\
&=  \norm{j^* (w_m^*-w_n^*)}_H \to 0 \quad\text{as }m,n\to \infty.
\end{align*}
Thus, $\big(i^*C^*S^*(T-\cdot)w_n^*\big)_{n\in\N}$ is a Cauchy sequence in $\mathcal M$ and therefore  it converges to some $g\in \mathcal M$. By the same arguments as in the proof of Proposition \ref{prop continuity of T:V* to M} it follows that there is a subsequence such that
$\lim_{k\to\infty} i^*C^*S^*(T-s)w_{n_k}^*= g(s)$ for Lebesgue almost all $s\in [0,T]$. On the other hand, the weak$^\ast$ convergence of $(w_{n_k}^\ast)_{k\in\N}$ implies that
$\lim_{k\to\infty}\scapro{ i^*C^*S^*(T-s)w_{n_k}^*}{x}=\scapro{ i^*C^*S^*(T-s)v^*}{x}$ for all $x\in X$ and
$s\in [0,T]$, which yields $g=i^*C^*S^*(T-\cdot)v^*$. It follows that
\begin{align*}
 \scapro{R w_{n_k}^*}{w_{n_k}^*}
=  \norm{i^*C^*S^*(t-\cdot)w_{n_k}^*}_{\mathcal M}^2
\to \norm{i^*C^*S^*(t-\cdot)v^*}_{\mathcal M}^2
= \scapro{Q v^*}{v^*},
\end{align*}
as $k\to\infty$. Therefore the covariance operators $R$ and $Q$ coincide on $V^\ast$, which yields that the cylindrical distribution of $Z_{\Psi}$ extends to a Radon measure.
\end{proof}

If $H>\tfrac{1}{2}$ then the first condition in Theorem \ref{th.existencesolution}, i.e.~$S(\cdot)C\in \mathcal I$, is satisfied for every strongly continuous semigroup as $L^2([0,T];X)\subseteq \M$. If $H<\frac{1}{2}$
this condition is not obvious but an important case is covered by the following result.
\begin{prop}\label{prop integrability of S(t)}
Let $H<\frac{1}{2}$. If $(S(t), t\ge 0)$ is an  analytic semigroup of negative type, then the mapping $S(\cdot)C$ is in $\mathcal I$.
\end{prop}

\begin{proof}
Define for arbitrary $f\in L^2([0,T];X)$
the function
\begin{align}\label{eq.defG}
 & G_f(s)\\
 &:=  b_H  \left(\frac{f(s)}{(T- s)^{\frac12-H}} + \left(\frac12-H\right) s^{\frac12-H} \int_s^T \frac{s^{H-\frac12}f(s)-t^{H-\frac12}f(t)}{(t-s)^{\frac32-H}} \,\ud t \right) \notag
\end{align}
for all $s\in (0,T]$. It follows from \eqref{eq.KastH<0.5} that a function $f\in L^2([0,T];X)$ is in $\mathcal M$ if and only
$G_f\in L^2([0,T];X)$, in which case $\norm{f}_{\mathcal M} =\norm{G_f}_{L^2}$. By the same computations as in the proof of \cite[Le.11.7]{pasik-duncan_et.al.06} one derives that there exists a constant $c_1>0$ such that
\begin{align}\label{eq.K*bound}
\int_0^T  \|G_f(s)\|^2 \,\ud s \nonumber
&\le  c_1\Bigg( \int_0^T \frac{\|f(s)\|^2}{(T- s)^{1-2H}} \,\ud s  +  \int_0^T   \frac{ \|f(s)\|^2}{s^{1-2H}} \, \ud s \\
 &\qquad\qquad  + \int_0^T \left(\int_s^T \frac{\|f(t)-f(s)\|}{(t-s)^{\frac32-H}}\, \ud t \right)^2  \, \ud s \Bigg).
\end{align}
We check that each term on the right hand side of \eqref{eq.K*bound} is finite for $f=i^* C^* S^*(\cdot)v^*$, $v^*\in V^*$. The growth bound of the semigroup guarantees
that there exist some constants $\beta,c_2>0$ such that
\begin{align*}
  \norm{S(s)}\le c_2 e^{-\beta s}
  \qquad\text{for  all } s\in [0,T].
\end{align*}
It is immediate that the first two integrals on the right hand
side in  \eqref{eq.K*bound} are finite since $1-2H<1$.
In order to estimate the last term, recall that, as $S$ is  analytic, there exists
for each $\alpha \ge 0$ a constant $c_3>0$ such that for every $0<s\le t$ we have
\[
\|S(t)- S(s)\| =\|(S(t-s)-\mathrm{Id} )S(s)\| \le c_3 (t-s)^\alpha s^{-\alpha} \mathrm e^{-\beta s}.
\]
Fix some $\alpha \in (\tfrac{1}{2}-H,\tfrac{1}{2})$.
The third summand on the right hand side of \eqref{eq.K*bound}
can be estimated by
\begin{align*}
\int_0^T & \left( \int_s^T \frac{\|i^*C^*S^*(t)v^*- i^*C^*S^*(s)v^*\|}{(t-s)^{\frac32-H}} \,\ud t \right)^2\,  \ud s\\
& \le  (c_3 \norm{i}\norm{C}\norm{v^\ast})^2\int_0^T \left(\int_s^T \frac{ \mathrm e^{-\beta s}}{s^{\alpha}} \frac{(t-s)^\alpha}{(t-s)^{\frac32-H}} \,\ud t \right)^2\,  \ud s\\
& \le (c_3 \norm{i}\norm{C}\norm{v^\ast})^2 \left( \int_0^T \frac{ \mathrm e^{-2\beta s}}{s^{2\alpha}} \, \ud s \right)\left(  \int_0^T \frac{ 1}{t^{\frac32 -H -\alpha}} \, \ud t \right)^2\\
&\le  (c_3 \norm{i}\norm{C}\norm{v^\ast})^2 (2\beta)^{2\alpha-1}\Gamma(1-2\alpha)
 T^{2(H+\alpha-\frac12 )},
\end{align*}
which completes the proof.
\end{proof}

Another example of a semigroup satisfying $S(\cdot)C$ in $\mathcal I$ is considered in Section
\ref{sc: application}. Further examples can be derived using the known fact that
the space of H{\"o}lder continuous functions of index larger than $\tfrac{1}{2}-H$ is continuously
embedded in $\mathcal M$.

\begin{example}
  If $V$ is a Hilbert space then a function $\Psi\in {\mathcal I}$ is stochastically integrable if and
  only if $\Gamma_\Psi$ is Hilbert-Schmidt, according to Theorem \ref{theo I_t honest process}. Thus,
  if $(f_k)_{k\in\N}$ denotes an orthonormal basis of $V$, the function $\Psi$ is stochastically
  integrable if and only if the adjoint operator $\Gamma^\ast_{\Psi}$ is Hilbert-Schmidt, that is
  \begin{align*}
    \sum_{k=1}^\infty \norm{\Gamma^\ast_{\Psi} f_k}^2_{L^2}
    =\sum_{k=1}^\infty \norm{K^\ast(i^\ast \Psi^\ast (\cdot)f_k)}^2_{L^2}
    =\sum_{k=1}^\infty \norm{i^\ast \Psi^\ast(\cdot)f_k}_{\mathcal M}^2<\infty.
  \end{align*}
  In the case $H>\tfrac{1}{2}$ we obtain that there exists a weak solution of \eqref{equation stoch Cauchy problem}  if
  \begin{align*}
 & \sum_{k=1}^\infty \int_0^T \int_0^T \norm{i^\ast C^\ast S^\ast (s) f_k} \norm{i^\ast C^\ast S^\ast (t)f_k}\abs{s-t}^{2H-2}\,\ud s \,\ud t <\infty.
  \end{align*}
For the case $H<\tfrac{1}{2}$ assume that the semigroup $(S(t),t\ge 0)$ is analytic and of negative type. A similar calculation as in the proof of Proposition \ref{prop integrability of S(t)} shows
that if there exists a constant $\alpha\in (\tfrac{1}{2}-H,\tfrac{1}{2})$ such
\begin{align*}
  \int_0^T \frac{\norm{S(s)Ci}_{HS}^2}{s^{2\alpha}}\,\ud s <\infty,
\end{align*}
then $S(\cdot)C$ is stochastically integrable. Here $\norm{\cdot}_{HS}$ denotes the Hilbert-Schmidt norm of an
operator $U:L^2([0,T];X)\to V$.

\end{example}

\section{Example: the stochastic heat equation}\label{sc: application}

As an example we consider a self-adjoint generator $A$ of a semigroup $(S(t),t\ge 0)$ in a separable Hilbert space $V$ such that there exists an orthonormal basis $(e_k)_{k\in\N}$ of $V$ satisfying $A e_k=-\lambda_k e_k$ for some $\lambda_k>0$ for all $k\in\N$ and $\lambda_k\to \infty$ as $k\to\infty$.  Thus the semigroup  satisfies
 \begin{align*}
   S(t)e_k=e^{-\lambda_k t} e_k \qquad\text{for all $t\ge 0$ and $k\in\N$.}
 \end{align*}
 A specific instance is the Laplace operator with Dirichlet boundary conditions
on $L^2(D;\R)$ for a set $D\in {\mathcal B}(\R^n)$.
We assume that $C=\id$ and we identify the dual space $V^\ast$ with $V$, i.e.
we consider the Cauchy problem
\begin{align}\label{eq.Cauchy-self-adjoint}
  \ud Y(t)= AY(t)\,\ud t + \ud B(t)
  \qquad\text{for all }t\in [0,T].
\end{align}
The system \eqref{eq.Cauchy-self-adjoint} is perturbed by a cylindrical fBm $B$ in $V$ which is independent along the orthonormal basis $(e_k)_{k\in\N}$ of eigenvectors $e_k$ of $A$,
that is we consider the cylindrical fBm $(B(t):\,t\ge 0)$ in $V$ from Example
\ref{ex.fBMheatequation}:
\begin{align*}
  B(t)v= \sum_{k=1}^\infty  \scapro{ie_k}{v} b_k(t)\qquad\text{for all }v\in V,
  \, t\ge 0,
\end{align*}
where $(b_k)_{k\in\N}$ is  a sequence of independent, real valued standard fBms of Hurst parameter $H\in (0,1)$ and the embedding $i:V\to V$ is defined by
\begin{align*}
  iv=\sum_{k=1}^\infty q_k \scapro{e_k}{v}e_k
\end{align*}
for a sequence $(q_k)_{k\in\N}\subseteq \R$  satisfying $\sup_k \abs{q_k}<\infty$. Note
that in this case $X=V$.

\begin{theo}\label{theorem classical solution in L^2}
Let $A$ be a self-adjoint generator satisfying the conditions described above.
If
\begin{align*}
  \sum_{k=1}^\infty \frac{q_k^2}{\lambda_k^{2H}}<\infty,
\end{align*}
then equation \eqref{eq.Cauchy-self-adjoint}  has a weak solution $(Y(t):\,t\in [0,T])$ in $V$. The solution can be represented by the variation of constants formula \eqref{eq.varofcons}.
\end{theo}

\begin{proof}
Note that in this situation we have
\begin{equation}\label{eq.normi*S*e_k}
i^\ast S^*(t) e_k = q_k  e^{-\lambda_k t}e_k
\qquad\text{for each $k\in\N$ and $t\in [0,T]$.}
\end{equation}
According to Theorem \ref{theo I_t honest process} and Theorem \ref{th.existencesolution}
we have to establish that $S(\cdot)$ is in $\mathcal I$ and the operator $\Gamma: L^2([0,T];V)\to V$ defined by
  \begin{align*}
 \scapro{\Gamma f}{v}= \int_0^T \big[K^*\big( i^\ast S^\ast(\cdot)v\big)(s), f(s)\big]\,\ud s
 \qquad\text{for all } f\in L^2([0,T];V),\, v\in V
\end{align*}
 is $\gamma$-radonifying. Since $V$ is a separable Hilbert space, the operator $\Gamma$ is $\gamma$-radonifying if and only if
it is Hilbert-Schmidt.

If $H>\tfrac{1}{2}$ then  $S(\cdot)$ is in $\mathcal I$ and
equality \eqref{eq.normi*S*e_k} yields for each $k\in\N$
\begin{align}\label{eq.aux22}
\norm{i^*S^*(\cdot) e_k}_{\M}^2
&= H(2H-1) \int_0^T \int_0^T \norm{i^*S^*(t) e_k}\norm{i^*S^*(s) e_k} \abs{s-t}^{2H-2} \,\ud s \,\ud t\nonumber\\
&= H(2H-1) q_k^2 \int_0^T  \int_0^T   e^{-\lambda_k t}  e^{-\lambda_k s} \abs{s-t}^{2H-2} \, \ud s  \,\ud t.
\end{align}
The iterated integral can be estimated by
\begin{align}\label{eq.aux22a}
\int_0^T  e^{-\lambda_k t} \int_0^T e^{-\lambda_k s} \abs{s-t}^{2H-2} \, \ud s  \,\ud t
&=2 \int_0^T  e^{-\lambda_k t} \int_0^t e^{-\lambda_k s} \abs{s-t}^{2H-2} \, \ud s  \,\ud t\notag\\
&=2 \int_0^T e^{-2\lambda_k t}\int_0^t e^{\lambda_k s} s^{2H-2}\,\ud s\, \ud t\notag\\
&=2 \int_0^T e^{\lambda_k s} s^{2H-2}\int_s^T e^{-2\lambda_k t}\,\ud t\, \ud s\notag\\
&\le \frac{1}{ \lambda_k} \int_0^T  e^{-\lambda_k s} s^{2H-2}\,\ud s\notag\\
&\le \frac{1}{\lambda_k^{2H}} \Gamma(2H-1).
\end{align}
Since inequality \eqref{eq.KisoMmod} guarantees that there
exists a constant $c>0$ such that
\begin{align*}
\sum_{k=1}^\infty \norm{\Gamma^\ast e_k}^2_{L^2}
= \sum_{k=1}^\infty \norm{K^\ast\big(i^* S^*(\cdot) e_k\big)}^2_{L^2}
\le c\sum_{k=1}^\infty  \norm{i^* S^*(\cdot) e_k}_{\M}^2,
\end{align*}
we can conclude from \eqref{eq.aux22} and \eqref{eq.aux22a} that
 $\Gamma^\ast$ and thus $\Gamma$ are Hilbert-Schmidt operators.

If $H<\tfrac{1}{2}$ Proposition \ref{prop integrability of S(t)} guarantees  that
$S(\cdot)$ is in $\mathcal I$.
As in the proof of Proposition \ref{prop integrability of S(t)} it follows that there exists a constant $c_1>0$ such that for all $k\in\N$
\begin{align}\label{eq.aux23}
\norm{K^\ast(i^\ast S^\ast(\cdot)e_k)}_{L^2}^2
& \le c_1 \Bigg(\int_0^T \frac{\norm{i^* S^*(s) e_k}^2 }{(T-s)^{1-2H} } \,\ud s  + \int_0^T \frac{\norm{i^* S^*(s) e_k}^2 }{s^{1-2H} } \,\ud s  \nonumber\\
&\quad+ \int_0^T \left( \int_s^T \frac{\norm{i^* S^*(t) e_k-i^* S^*(s) e_k} }{(t-s)^{\frac32-H} } \mathrm d t \right)^2 \,\ud s\Bigg).
\end{align}
Equality \eqref{eq.normi*S*e_k} implies for the
first integral  the estimate
\begin{align}\label{eq.aux231}
 \int_0^T \frac{\norm{i^* S^*(s) e_k}^2 }{(T-s)^{1-2H} } \,\ud s
&\le    q_k^2\int_0^T \frac{e^{-2\lambda_ks}}{(T-s)^{1-2H}}\,\ud s\notag\\
&=\frac{q_k^2}{(2\lambda_k)^{2H}}\int_0^{2\lambda_k T}\frac{e^{-s}}{(2\lambda_k T-s)^{1-2H}}\,\ud s\notag\\
&\le\frac{q_k^2}{(2\lambda_k)^{2H}} \left(1+\frac{1}{2H}\right).
\end{align}
Here, the estimate of the integral follows from the fact that if $2\lambda_k T\le 1$ then
\begin{align*}
  \int_0^{2\lambda_k  T} \frac{e^{-s}}{(2\lambda_k  T-s)^{1-2H}}\,\ud s
  \le   \int_0^{2\lambda_k  T} \frac{1}{(2\lambda_k  T-s)^{1-2H}}\,\ud s\le \frac{1}{2H},
\end{align*}
and if $ 2\lambda_k T >1 $ then
\begin{align*}
  \int_0^{2\lambda_k T } \frac{e^{-s}}{(2\lambda_k  T-s)^{1-2H}}\,\ud s
&\le \int_0^{2\lambda_k T-1} e^{-s}\,\ud s +  \int_{2\lambda_k T-1}^{2\lambda_k T} (2\lambda_k T-s)^{2H-1}\,\ud s\\
&\le 1 + \frac{1}{2H}.
\end{align*}
The second integral in \eqref{eq.aux23} can be bounded by
\begin{align}\label{eq.aux232}
\int_0^T \frac{\norm{i^* S^*(s) e_k}^2 }{s^{1-2H} } \,\ud s
\le  q_k^2 \int_0^T \frac{e^{-2\lambda_ks}}{s^{1-2H}}\,\ud s
\le \Gamma(2H) \frac{q_k^2}{(2\lambda_k)^{2H}}.
\end{align}
Another application of equality \eqref{eq.normi*S*e_k} yields for the third term in \eqref{eq.aux23}
\begin{align}\label{eq.aux233}
&\int_0^T \left( \int_s^T \frac{\norm{i^* S^*(t) e_k -i^* S^*(s) e_k} }{(t-s)^{\frac32-H} } \,\ud t \right)^2 \,\ud s\notag\\
&\qquad\qquad=
q_k^2 \int_0^T \left( \int_s^T \frac{\abs{e^{-\lambda_k t}-e^{-\lambda_k s}}}{(t-s)^{\frac32-H} } \,\ud t \right)^2 \,\ud s\notag\\
&\qquad\qquad =
q_k^2 \int_0^T  e^{-2\lambda_k s} \left( \int_0^{T-s} \frac{1-e^{-\lambda_k t}}{t^{\frac32-H} } \,\ud t \right)^2 \,\ud s .
\end{align}
Applying the changes of variables $\lambda_k s=x$ and $\lambda_k t=y$ yields
\begin{align}\label{eq.aux233.1}
& \int_0^T  e^{-2\lambda_k s} \left( \int_0^{T-s} \frac{1-e^{-\lambda_k t}}{t^{\frac32-H} } \,\ud t \right)^2 \,\ud s\notag\\\
&\qquad\qquad = \frac{1}{\lambda_k^{2H}} \int_0^{\lambda_k T} e^{-2x} \left(\int_0^{\lambda_k T-x}   \frac{1-e^{-y}}{y^{\frac32-H}}\,\ud y\right)^2\, \ud x\notag\\
&\qquad\qquad = \frac{1}{\lambda_k^{2H}} \int_0^{\lambda_k T} e^{-2(\lambda_k T-x)} \left(\int_0^{x}  \frac{1-e^{-y}}{y^{\frac32-H}}\,\ud y\right)^2\, \ud x\notag\\
&\qquad\qquad \le \frac{1}{\lambda_k^{2H}}  c_2,
\end{align}
where $c_2>0$ denotes a constant only depending on $H$ but not on $\lambda_k$.
The finiteness of the constant $c_2$ and its independence of $\lambda_k$ follow from the
following three estimates:
\begin{align*}
  \int_0^{1} e^{-2(\lambda_k T-x)} \left(\int_0^{x}  \frac{1-e^{-y}}{y^{\frac32-H}}\,\ud y\right)^2\, \ud x
&  \le   \int_0^{1} \left(\int_0^{1} \frac{1-e^{-y}}{y^{\frac32-H}}\,\ud y\right)^2\, \ud x, \\
  \int_1^{\lambda_k T} e^{-2(\lambda_k T-x)} \left(\int_0^{1}  \frac{1-e^{-y}}{y^{\frac32-H}}\,\ud y\right)^2\, \ud x
 & \le  \frac{1}{(H+\tfrac12)^2} \int_1^{\lambda_k T} e^{-2(\lambda_k T-x)} \, \ud x,\\
  \int_1^{\lambda_k T} e^{-2(\lambda_k T-x)} \left(\int_1^{x}  \frac{1-e^{-y}}{y^{\frac32-H}}\,\ud y\right)^2\, \ud x
 & \le  \frac{1}{(H-\tfrac12)^2} \int_1^{\lambda_k T} e^{-2(\lambda_k T-x)} x^{2H-1} \ud x\\
  & \le  \frac{1}{(H-\tfrac12)^2} \int_1^{\lambda_k T} e^{-2(\lambda_k T-x)}\,  \ud x .
\end{align*}
By applying the estimates  \eqref{eq.aux231}--\eqref{eq.aux233.1} to \eqref{eq.aux23}, it follows that  there exists a constant $c_3>0$ such that
\begin{align*}
 \norm{\Gamma^\ast e_k}_{L^2}^2= \norm{K^\ast(i^\ast S^\ast(\cdot)e_k)}_{L^2}^2\le c_3 \frac{q_k^2}{\lambda_k^{2H}}
  \qquad\text{for all }k\in\N.
\end{align*}
As before we can conclude that $\Gamma$ is Hilbert-Schmidt.
\end{proof}

Consider now the special case of the heat equation with Dirichlet boundary conditions driven by a cylindrical fractional noise with independent components, that is with $Q=\id$. In this case $q_k\equiv 1$ and the eigenvalues of the Laplacian  behave like $\lambda_k\thicksim k^{2/n} $ so that the condition for the existence of a weak solution becomes the well known $n/4<H<1$. This result is in line with the literature, see for example \cite{brzezniak_et.al.12, duncan_et.al.02, maslowski_nualart03}.

\bibliographystyle{plain}
\bibliography{cylindrical-fBm-biblio}


\end{document}